\documentclass{article}
\usepackage{amsmath,amsfonts,amssymb,bm,booktabs,graphicx}
\usepackage{amsthm}
\usepackage{algorithm, algpseudocode}
\usepackage[utf8]{inputenc}
\usepackage[left=2cm,right=2cm,top=2cm,bottom=2cm]{geometry}
\usepackage{subfig}
\usepackage{float}
\usepackage{arydshln}
\usepackage{bbm}

\newtheorem{theorem}{Theorem}
\newtheorem{lemma}{Lemma}
\newtheorem{corollary}{Corollary}
\providecommand{\keywords}[1]{\textbf{\textit{Keywords---}} #1}

\newcommand{\bfb}{\bm{b}}
\newcommand{\bfe}{\bm{e}}

\newcommand{\bfx}{\bm{x}}

\newcommand{\bfphi}{\bm{\phi}}
\newcommand{\bfpsi}{\bm{\psi}}

\newcommand{\calR}{{\mathcal{R}}}

\newcommand{\bilintilbothnopar}[4]{ a_{#1 \, #2} \left( e_W^{\{#3 + #1/p\}},\phi^{[#4 + #2/p]} \right)}
\newcommand{\bilintilfirstnopar}[4]{ a_{#1 \, #2} \left( e_W^{\{#3 + #1/p\}},\phi^{[#4 + (#2)/p]} \right)}
\newcommand{\bilintilsecondnopar}[4]{ a_{#1 \, #2} \left( e_W^{\{#3 + (#1)/p\}},\phi^{[#4 + #2/p]} \right)}
\newcommand{\bilinhate}[3]{a_{#1}\left(e_D^{\{#2 + (#1)/p\}}, \phi^{[#3 + (#1)/p]}\right)}
\newcommand{\bilinhatenopar}[3]{a_{#1}\left(e_D^{\{#2 + #1/p\}}, \phi^{[#3 + #1/p]}\right)}

\title{\emph{A posteriori} error analysis for Schwarz overlapping domain decomposition methods}
\author{Jehanzeb Chaudhry, Don Estep and Simon Tavener}
\date{}

\begin{document}
\maketitle{}
%\hfill Date \& time: \timestamp

%SJT: This is the abstract for my seminar talk and may need some modification
\begin{abstract}
Domain decomposition methods are widely used for the numerical solution of partial differential equations on high performance computers. We develop an adjoint-based \emph{a posteriori} error analysis for both multiplicative and additive overlapping Schwarz domain decomposition methods. The numerical error in a user-specified functional of the solution (quantity of interest) is decomposed into contributions that arise as a result of the finite iteration between the subdomains and from the spatial discretization. The spatial discretization contribution is further decomposed into contributions arising from each subdomain. This decomposition of the numerical error is used to construct a two stage solution strategy that efficiently reduces the error in the quantity of interest by adjusting the relative contributions to the error.
\end{abstract}

\keywords{Schwarz overlapping domain decomposition, \emph{a posteriori} error estimation, adaptive computation}

% \begin{keywords}
% Schwarz overlapping domain decomposition, \emph{a posteriori} error estimation, adaptive computation
% \end{keywords}

% \begin{AMS}
% 65N55, 65N15, 68W10, 65N50
% \end{AMS}

\section{Introduction}
\label{sec:introduction}

We derive and implement an adjoint-based \emph{a posteriori} error analysis for overlapping domain decomposition methods for elliptic boundary value problems, examining both additive and multiplicative Schwarz algorithms. Domain decomposition methods (DDMs) arrive at the solution of a problem defined over a domain by combining the solutions of related problems posed on subdomains. The problems posed on subdomains often require less computational resources and some of the first uses of DDMs for practical applications were in low-memory or limited computation scenarios~\cite{Kron1953,Przemieniecki63}.
%\footnote{The first mathematical formulation of DDMs dates much further back to Schwarz~\cite{Schwarz:1870:UGA} who introduced the multiplicative overlapping DDM in 1870. Schwarz constructed solutions of partial differential equations (PDEs) in complicated geometries by decomposing the domain into simpler shapes on which the solution could be found analytically (e.g. by using separation of variables) and then defined an iteration which converged to the true solution under suitable conditions.}
Recently DDMs have seen increased use in the context of distributed and parallel computing~\cite{Keyes:1995:DBP, Tarek2008, Smith:1996:DPM, Toselli:2004:DDM, Wohlmuth:1999:DMI}. In this article, we follow the presentation in \cite{Tarek2008}.

In overlapping DDMs, each subdomain has a non-empty intersection with at least one other subdomain and typically  state information is exchanged between the subdomains. The theoretical properties of the multiplicative Schwarz method and some of its variants were studied in~\cite{Lions:1988:SAM}. The variant of this method suitable for parallel computing, called the additive Schwarz method, was introduced in~\cite{Dryja:1987:AVS}.
%An excellent historical perspective of Schwarz methods may be found in~\cite{Gander:2008}.
Non-overlapping DDMs, in which the  subdomains have empty intersection and state and derivative information is exchanged through their common interfaces, is an alternative approach~\cite{lions1990schwarz}.

%The first  non-overlapping method was introduced in \cite{lions1990schwarz}, and an \emph{a posteriori} analysis of this method was presented in \cite{JBE18}. There are numerous other variants of non-overlapping methods e.g. Schur-complement and iterative substructuring~\cite{AN1997, GGQ96, QV2008} and Lagrange multiplier based substructuring methods~\cite{FLLPR2001, FMR1994, FaRo1991, MaDo2003}.

% the most well-known of which are the FETI~\cite{FaRo1991,FMR1994} and the BDDC methods~\cite{FLLPR2001,MaDo2003}.

Adjoint-based \emph{a posteriori} error analysis for systems of ordinary and partial differential equations has an extensive history \cite{bangerth2013adaptive, Becker:01, eriksson1995introduction,  Estep:95, Giles:Suli:02, houston2002hp}, and has been applied to a wide range of applications and numerical methods~\cite{arbogast2014posteriori, arbogast2015posteriori, JBE18,  CEG+2015, JEG16,  Chaudhry2019, Chaudhry17, chaudhry2016posteriori, collins2014_explicit, collins2015_cutcell, collins2014_LaxWendroff, johansson2015adaptive}.
Adjoint-based \emph{a posteriori} error analysis classically considers a differential equation
\begin{equation}
\label{eq:diff_eq_generic}
  L(u (\mathbf{x}, t)) = g(\mathbf{x}, t),
\end{equation}
where $L$ denotes the differential operator, and a Quantity of Interest (QoI), expressed as a linear functional
\begin{equation}
\label{eq:QoI_generic}
  Q(u)=(u,\psi),
\end{equation}
where $(\cdot, \cdot)$ denotes the $L_2$ inner product and the function $\psi$ is chosen to yield the desired information. Given the numerical approximation $U$ to the analytical solution, the residual $R(U) = g - L(U)$ quantifies the effects of discretization on the evaluation of the differential equation, but it does not provide the error in the QoI, $(u-U,\psi)$. The relation between the residual and the error is derived from solving an adjoint problem.

For linear problems, the adjoint operator $\mathcal{L}^\ast : Y^\ast \rightarrow X^\ast$ of a linear operator $\mathcal{L}: X \rightarrow Y$ between Banach spaces $X,Y$  with dual spaces $X^\ast, Y^\ast$ is defined by the bilinear identity
\begin{equation}
\label{eq:bilin}
\langle \mathcal{L} x,y^\ast \rangle_Y = \langle x, \mathcal{L}^\ast y^\ast \rangle_X, \quad x \in X,   y^\ast \in Y^\ast,
\end{equation}
where  $\langle\cdot, \cdot \rangle_S$ denotes duality-pairing in the space $S \in \lbrace X, Y \rbrace$.
The adjoint problem associated with \eqref{eq:diff_eq_generic} is
\begin{equation}\label{formaladjoint}
 L^\ast \phi = \psi.
 \end{equation}
This yields the error estimate,
\begin{equation}
\label{eq:generic_err_eq_simple}
\text{Error in the QoI } = (u - U, \psi) = (R(U), \phi).
 \end{equation}
We estimate the numerical error in the quantity of interest by numerically solving the adjoint problem \eqref{formaladjoint}, computing the residual, and evaluating \eqref{eq:generic_err_eq_simple}. % Scaling residuals by the adjoint solution accounts for the accumulation, cancellation, and propagation of local contributions to the error. The stability information is specific to the quantity of interest.

Classical \emph{a posteriori} error analysis for the numerical solution of differential equations assumes the use of fully implicit discretization methods in which the approximate solution is computed exactly for which the adjoint of the forward operator \eqref{formaladjoint} produces a useful adjoint solution. The adjoint of the discrete solution operator when implementing more complex, multistage solution methods is much more complicated to define.  If the steps in the solution process are written as compositions of operators, then the appropriate adjoint can typically be written as a composition of adjoints associated with various steps of discretization. The resulting error estimate must then use the appropriate adjoint to weight specific residuals and include additional terms quantifying the difference between this adjoint and the adjoint of the overall problem \eqref{formaladjoint}. The correct choice of adjoint and residuals also enables a decomposition of the total error into distinct sources of error, such as discretization, iteration, transfer, projection and quadrature errors. These concepts are illustrated in an analysis of iterative solvers for non-autonomous evolution problems in \cite{CEGT13b}, in an analysis of a multirate iterative solver for ordinary differential equations in \cite{ginting2012multirate}, and in an analysis of an iterative multi-discretization method for reaction-diffusion systems in \cite{CEG+2013}. An \emph{a posteriori} error analysis for non-overlapping DDM is carried out in~\cite{JBE18}. To the best of our knowledge, an \emph{a posteriori} error analysis for overlapping DDMs has not been performed.

Adjoint-based \emph{a posteriori} error estimates can provide useful information for designing efficient adaptive solution strategies. During the first ``pre-processing'' step (stage 1), a solution is computed on a relatively coarse discretization together with an accurate \emph{a posteriori} error estimate that quantifies the contributions of all sources of  error.  The information provided by the first stage is used to guide discretization choices  for a ``production level'' (stage 2) computation.
%An examination of the relative, as well as absolute sizes of the various  contributions to the total error, enables the discretization strategy for a production level run to be chosen so as to approximately balance the sources of error and achieve a given tolerance.
This strategy is described in earlier work on blockwise adaptivity \cite{carey2010blockwise, johansson2015adaptive} and in \cite{chaudhry2016posteriori}. %A two-stage computational approach is also adopted here. % Multiplicative and additive Schwarz methods are closely related to Gauss-Seidel and Jacobi iterations for linear systems of algebraic equations. In an attempt to provide additional insight, an \emph{a posteriori} error analysis for Gauss-Seidel iteration is included as an appendix.

We introduce the multiplicative and additive Schwarz overlapping domain decomposition methods in \S \ref{sec:Schwarz_DD}. We present the \emph{a posteriori} error analysis in \S \ref{sec:Schwarz_error_analysis}.
%Definitions of discretization and iteration errors appear in \S \ref{sec:Schwarz_error_analysis}, as well as the adjoint problems and error representation formulas for both multiplicative and additive Schwarz.
Examples are provided for multiplicative Schwarz in \S \ref{sec:mult_Schwarz_numerical_examples} and for additive Schwarz in \S \ref{sec:add_Schwarz_numerical_examples}. Details of the analysis appear in \S \ref{sec:proofs}. A discussion and future research directions appear in \S \ref{sec:discussion}.

\section{Overlapping Schwarz domain decomposition}
\label{sec:Schwarz_DD}

Assume that we have $p$ overlapping subdomains $\Omega_1, \cdots, \Omega_p$ on a domain $\Omega$, such that for any  $\Omega_i$, there exists a $\Omega_j$, $i \neq j$, for which $\Omega_i \cap \Omega_j \neq \emptyset$ and $\cup_i \Omega_i = \Omega$. We use $L_2(\Omega)$  to denote the space of square integrable functions, $H^1(\Omega)$ for functions having a $L_2(\Omega)$ derivative and  $H^1_0(\Omega)$ as the subspace of $H^1(\Omega)$ of functions satisfying homogeneous Dirichlet boundary conditions. We let $(\cdot, \cdot)$ and  $(\cdot, \cdot)_{ij}$ represent the $L_2(\Omega)$ and  $L_2(\Omega_i \cap \Omega_j)$ inner products respectively.

Consider the weak form of a second-order linear elliptic partial differential equation (PDE) problem: find $u \in H^1_0(\Omega)$ such that
\begin{equation}
\label{eq:abstract_weak_form}
	a(u,v) = l(v) \quad \forall v \in H^1_0(\Omega).
\end{equation}
Here $a(\cdot, \cdot)$ is the standard bilinear form over $\Omega$ arising from integration by parts of the PDE operator and $l(\cdot)$ is the linear functional arising from the right-hand-side of the PDE. For example, given the Poisson equation $-\nabla^2 u(x) =  f(x)$ with homogeneous Dirichlet boundary conditions, we have $a(u,v) = \int_{\Omega} \nabla u \cdot \nabla v \, dx$ and $l(v) = (f,v)$. We denote by  $a_i(\cdot, \cdot)$ the  restriction  of $a(\cdot, \cdot)$  to  $\Omega_i$ and $a_{ij} (\cdot, \cdot)$  the restriction of $a(\cdot, \cdot)$ to $\Omega_i \cap \Omega_j$. Similarly, we let $l_i(\cdot)$ be the restriction of $l(\cdot)$ to $\Omega_i$.

We are interested in a QoI which is a linear functional of the solution,
\begin{equation}
\label{eq:QoI_def}
Q(u) = (\psi,u),
\end{equation}
where $\psi \in L_2(\Omega)$.

\subsection{Multiplicative Schwarz overlapping domain decomposition}
\label{sec:multiplicative_Schwarz_algorithms}

Defining $ H^{1}_{D_{k}}(\Omega_i) \equiv \{v \in H^{1}(\Omega_i) |\ v = u^{\{k+(i-1)/p\}} \text{ on } \partial \Omega_i\}$, we present the multiplicative Schwarz method in Algorithm \ref{alg:multiplicative_basic}.

\begin{algorithm}[H]
\caption{Overlapping multiplicative Schwarz domain decomposition}
\label{alg:multiplicative_basic}
\begin{algorithmic}
\State Given $u^{\{0\}}$ defined on $\Omega$
\For{$k=0, 1, 2, \dots, K-1$ }
  \For{$i= 1, 2, \dots, p$ }
    \State Find $\widetilde{u}^{\{k+i/p\}} \in H^{1}_{D_{k}}(\Omega_i)$ such that \\ \qquad\qquad
    \begin{equation}
      \label{eq:multiplicative_basic_local}
	    a_i \big( \widetilde{u}^{\{k+i/p\}},v \big)= l_i(v),
            \quad \forall v \in H^{1}_{0}(\Omega_i).
    \end{equation}
   % \State \\
    \State Let
    \begin{equation}
    \label{eq:multiplicative_basic_global}
    u^{\{k+i/p\}} =
      \left \{
      \begin{gathered} \begin{aligned}
      &\widetilde{u}^{\{k + i/p\}}, \; &&\hbox{on } \overline{\Omega}_i, \\
      &u^{\{k + (i-1)/p\}},            &&\hbox{on } \Omega \backslash \overline{\Omega}_i.
      \end{aligned} \end{gathered}
      \right .
      \end{equation}
    % \State \\
  \EndFor
\EndFor
\end{algorithmic}
\end{algorithm}

% and $V_i \subseteq H^{1}_{0}(\Omega_i)$.

\subsection{Additive Schwarz overlapping domain decomposition}
\label{sec:additive_Schwarz_algorithms}

%We derive an adjoint-based \emph{a posteriori} error analysis for additive overlapping domain decomposition (Jacobi) after $K$ iterations.
%
% \subsection{Algorithm and notation}
% Assume that we have $p$ overlapping subdomains $\Omega_1, \cdots, \Omega_p$ on a domain $\Omega$. Let $a(\cdot, \cdot)$ be a bilinear form over $\Omega$, $a_i(\cdot, \cdot)$ its restriction over $\Omega_i$ and $a_{ij} (\cdot, \cdot)$ its restriction over $\Omega_i \cap \Omega_j$. Let $(\cdot, \cdot)$ and  $(\cdot, \cdot)_{ij}$ represent the $L_2(\Omega)$ and  $L_2(\Omega_i \cap \Omega_j)$ inner products respectively.
%
% \subsection{Algorithm for additive Schwarz}
%
The additive Schwarz solution method is given in Algorithm \ref{alg:additive_basic} with $H^{1}_{D_{k}}(\Omega_i) \equiv \{v \in H^{1}_D(\Omega_i) |\ v = u^{\{k\}} \text{ on } \partial \Omega_i\}$. The  Richardson parameter $\tau$, is needed to ensure that the iteration converges~\cite{Tarek2008}.

\begin{algorithm}[H]\caption{Overlapping additive Schwarz domain decomposition}
\label{alg:additive_basic}
\begin{algorithmic}
\State Given $u^{\{0\}}$ defined on $\Omega$
\For{$k=0, 1, 2, \dots, K-1$ }
\For{$i= 1, 2, \dots, p$ }
% \State \\
\State Find $\widetilde{u}_i^{\{k\}} \in H^{1}_{D_{k}}(\Omega_i)$ such that \\ \qquad\qquad
  \begin{equation} \label{eq:additive_basic_local}
    a_i \big( \widetilde{u}_i^{\{k+1\}},v \big)= l_i(v), \quad \forall v \in H^{1}_{0}(\Omega_i).
  \end{equation}
% \State \\
\State Let
  \begin{equation} \label{eq:additive_basic_global}
     u^{\{k+1\}} = (1 - \tau p)u^{\{k\}}
                 + \tau \left ( \sum_{i=1}^p  {\Pi_i} \widetilde{u}_i^{\{k+1\}} \right )
     \;\hbox{where}\quad
     \Pi_i \widetilde{u}_i^{\{k+1\}} = \left \{
    \begin{gathered} \begin{aligned}
      &\widetilde{u}_i^{\{k+1\}}, \; &&\hbox{on } \overline{\Omega}_i, \\
      &u^{\{k \}},                   &&\hbox{on } \Omega \backslash \overline{\Omega}_i.
    \end{aligned} \end{gathered}  \right .
  \end{equation}
%
% \State \\
\EndFor
\EndFor
\end{algorithmic}
\end{algorithm}

%%%%%%%%%%%%%%%%%%%%%%%%%%%%%%%%%%%%%%%%%%%%%%%%%%%%%%%%%%%%%%%%%%%%%%%%%%%%%%%%%%%%%%%%%%%%%%%%%%%%
%                                                                                                  %
%   Finite element discretization                                                                  %
%                                                                                                  %
%%%%%%%%%%%%%%%%%%%%%%%%%%%%%%%%%%%%%%%%%%%%%%%%%%%%%%%%%%%%%%%%%%%%%%%%%%%%%%%%%%%%%%%%%%%%%%%%%%%%

\subsection{Finite element discretizations}
\label{sec:finite_element_discretization}

We let $\mathcal{T}_h = \{T_m \}_{m=1}^{M}$ denote a quasi-regular triangulation of $\Omega$ in to non-overlapping elements $T_m$ such that no node of one element $T_i$ intersects an interior edge of another element $T_{j}$, and $\Omega = \cup_m T_m$. Moreover,  the triangulation is consistent with the domain decomposition  in the sense that if $T_i \cap \Omega_j \neq \emptyset$ then $T_i \subset \Omega_j$. The discretization of the overlapping domain decomposition approximation  substitutes finite dimensional spaces $V_{i,h}^{k}$ for $H^{1}_{D_{k}}(\Omega_i)$ and $V_{i,h,0}$ for $H^1_0(\Omega_i)$ in Algorithm~\ref{alg:multiplicative_basic},  where $V_{i,h}^{k}$ and $V_{i,h,0}$ refer to the standard finite element spaces consisting space of continuous piecewise polynomial functions on $\mathcal{T}_{h,i} = \mathcal{T}_h\vert_{\Omega_i}$. Additionally, $V_h \subset H^1_0(\Omega)$ is the  finite element space consisting  of continuous piecewise polynomial functions with respect to $\mathcal{T}_h$.

We represent the global discrete solutions as $U^{\{k+i/p\}}$ (resp. $U^{\{k\}}_i$) belonging to the space $V_{h}$ and the local discrete solutions as $\widetilde{U}^{\{k+i/p\}}$  (resp. $\widetilde{U}^{\{k\}}_i$)  belonging to the space $V_{i,h,0}$ for the multiplicative (resp. additive) Schwarz methods. For simplicity we assume that $U^{\{0\}} = u^{\{0\}} $. For both algorithms, the global continuum, (resp. discrete), solution after $k$ iterations is represented as $u^{\{k\}}$, (resp. $U^{\{k\}}$).

\section{\emph{A posteriori} analysis for the finite element approximation computed using Schwarz algorithms}
\label{sec:Schwarz_error_analysis}

We derive a representation formula for the error in the QoI, $Q(u) - Q(U^{\{K\}}) = \left(\psi,u - U^{\{K\}}\right)$,  that is computed from the discrete solution of the multiplicative or additive domain decomposition method  after $K$ iterations.

%That is, we use adjoint-based analysis to find an error representation for $Q(u) - Q(U^{\{K\}}) = (\psi,u - U^{\{K\}})$.

% \subsection{Preliminaries}

% Let $\psi \in L_2(\Omega)$. The QoI is represented as,
% \begin{equation}
% \label{eq:QoI_def}
% Q(u) = (\psi,u).
% \end{equation}
% The aim is to find the error in the QoI computed from the discrete solution of the  domain decomposition method (multiplicative or additive)  after $K$ iterations.

%%%%%%%%%%%%%%%%%%%%%%%%%%%%%%%%%%%%%%%%%%%%%%%%%%%%%%%%%%%%%%%%%%%%%%%%%%%%%%%%%%%%%%%%%%%%%%%%%%%%
%                                                                                                  %
%   Discretization and iteration errors                                                            %
%                                                                                                  %
%%%%%%%%%%%%%%%%%%%%%%%%%%%%%%%%%%%%%%%%%%%%%%%%%%%%%%%%%%%%%%%%%%%%%%%%%%%%%%%%%%%%%%%%%%%%%%%%%%%%

\subsection{The total error and its components}
We first consider the total numerical error, and then its decomposition into  discretization and iteration error components.

\subsubsection{The total error}
\label{sec:multiplicative_Schwarz_total_iter_errors}

We define the \emph{global} adjoint
\begin{equation}
\label{eq:global_adj}
	a(v,\phi) = (\psi,v), \quad \forall v \in H^1_0(\Omega).
\end{equation}
% and obtain a representation for total error as below.

\begin{theorem}[Total error representation]
\label{thm:multiplicative_total_error}
The error in the QoI for the discretized multiplicative or additive Schwarz algorithm after $K$ iterations is given by
\begin{equation}
\label{eq:multiplicative_tot_err_rep}
\left(u - U^{\{K\}}, \psi\right) = R \left(U^{\{K\}}, \phi\right),
\end{equation}
where $R\left(U^{\{K\}}, \phi\right) = l(\phi) - a\left(U^{\{K\}}, \phi\right)$ is the weak residual.
\end{theorem}

The proof of Theorem \ref{thm:multiplicative_total_error} is standard, see e.g., \cite{eriksson1995introduction}. Unfortunately, it does not capture the structure of the differential operator corresponding to the Schwarz domain decomposition, which is reflected in the lack of Galerkin orthogonality in the expression. Performing Schwarz domain decomposition with a finite number of iterations defines a differential operator which is different than the differential operator associated with the original PDE~\eqref{eq:abstract_weak_form}. The numerical solution $U^{\{K\}}$ is a solution to the discretization of this modified operator.  We carry out an analysis by decomposing the error into two contributions: iterative and discretization errors. For implementation purposes we note that the global adjoint $\phi$ is solved using a higher order finite element scheme. The global adjoint may be approximated by a Schwarz domain decomposition method provided sufficient iterations are performed, or the overlap is sufficiently large that the iteration error is negligible.

\subsubsection{Discretization and iteration errors}
\label{sec:multiplicative_Schwarz_disc_iter_errors}

%We assume that the initial iterates for both the FEM solution and the analytical solutions are the same. That is, $u^{\{0\}} = U^{\{0\}}$.
We decompose the total error as
\begin{equation}
\label{eq:error_decomposition}
u - U^{\{K\}}
= \underbrace{ u - u^{\{K\}}}_{\text{Iteration Error}}
+ \underbrace{  u^{\{K\}} -  U^{\{K\}}}_{\text{Discretization Error}}
= e_I^{\{K\}} + e_D^{\{K\}},
\end{equation}
where  $e_I^{\{k\}} = u - u^{\{k\}}$, $e_D^{\{k\}} = u^{\{k\}} - U^{\{k\}}$ and  $e_D^{\{0\}} = 0$. The iteration error quantifies the error due to the discrepancy between the PDE differential operator and the modified differential operator in the Schwarz algorithms arising from using a finite number ($K$) iterations. The discretization error arises from the discretization of  the modified differential operator.

%%%%%%%%%%%%%%%%%%%%%%%%%%%%%%%%%%%%%%%%%%%%%%%%%%%%%%%%%%%%%%%%%%%%%%%%%%%%%%%%%%%%%%%%%%%%%%%%%%%%
%                                                                                                  %
%   Total and Iteration Error Analysis                                                             %
%                                                                                                  %
%%%%%%%%%%%%%%%%%%%%%%%%%%%%%%%%%%%%%%%%%%%%%%%%%%%%%%%%%%%%%%%%%%%%%%%%%%%%%%%%%%%%%%%%%%%%%%%%%%%%

\begin{theorem}[Iteration error representation]
We have
\begin{equation}
\label{eq:multiplicative_it_err_rep}
\left(u - u^{\{K\}}, \psi\right) = R \left(U^{\{K\}}, \phi\right)
   - \left(\psi, u^{\{K\}} - U^{\{K\}}\right).
\end{equation}
\end{theorem}

\begin{proof}
This follows by combining \eqref{eq:multiplicative_tot_err_rep} and \eqref{eq:error_decomposition}.
\end{proof}

%The analysis of the discretization error in the QoI, $(\psi, u^{\{K\}} - U^{\{K\}})$, is more involved and we analyze this for the multiplicative Schwarz algorithm in \S \ref{sec:multiplicative_Schwarz_error_analysis} and for the additive Schwarz algorithm in \S \ref{sec:additive_Schwarz_error_analysis}.
The analysis involves partitioning of the QoI data over subdomains by a partition of unity. Similar ideas were used in \cite{EHL2005}. Let $\{\chi_i \}_{i=1}^{p}$ be a partition of unity such that
\begin{equation}
\label{eq:pou}
\psi_i=\chi_i \psi,
\end{equation}
and $\psi_i=0$ on $\Omega \backslash \Omega_i$. The partition of unity localizes the QoI data to the subdomains.  Let $d_i(x)$ denote the distance function
\begin{equation}
	d_i(x) = \left \{ \begin{aligned}
\mathrm{dist}(x, B^{(i)}), &\text{ if } x \in \overline{\Omega}_i\\
0, &\text{ if } x \notin \overline{\Omega}_i,
	\end{aligned}
	\right.
\end{equation}
where $B^{(i)} \equiv  (\partial \Omega_i \cap \Omega)$. Then set
\begin{equation}
	\chi_i(x) = \frac{d_i(x)}{\sum_{j=1}^{p} d_j(x)} \,.
\end{equation}

%%%%%%%%%%%%%%%%%%%%%%%%%%%%%%%%%%%%%%%%%%%%%%%%%%%%%%%%%%%%%%%%%%%%%%%%%%%%%%%%%%%%%%%%%%%%%%%%%%%%
%                                                                                                  %
%  Lemma Zero                                                                                       %
%                                                                                                  %
%%%%%%%%%%%%%%%%%%%%%%%%%%%%%%%%%%%%%%%%%%%%%%%%%%%%%%%%%%%%%%%%%%%%%%%%%%%%%%%%%%%%%%%%%%%%%%%%%%%%

\bigskip
With the partition of the QoI data, we have the following partition of the QoI.
\begin{lemma}[Partitioning the QoI data over subdomains]
\label{lem:partition of unity}
We have
\begin{equation}
\label{eq:multiplicative_err_pou_decomp}
\left( e_D^{\{k\}},\psi \right)
= \sum_{i=1}^p \left( e_D^{\{k\}},\psi_i \right)_{ii}.
\end{equation}

\end{lemma}

\begin{proof}
This follows directly from the definition of the partition of unity in \eqref{eq:pou},
\[
\left( e_D^{\{k\}},\psi \right)
= \left( e_D^{\{k\}},\sum_{i=1}^p \chi_i \psi \right)	
= \sum_{i=1}^p \left( e_D^{\{k\}},\psi_i \right)_{ii}.
\]
\end{proof}

\subsubsection{Weak Residuals}
 We define the  subdomain weak residual for a function $s$,
\begin{equation}
\label{eq:weak_resid_subdomain}
	R_i(s,v) = l_i(v) - a_i(s,v),
\end{equation}
for $i = 1,2, \ldots, p$.

\subsection{A posteriori error analysis of discretization error for multiplicative Schwarz}

In this section, we derive a representation of the discretization error, $\left(\psi, u^{\{K\}} - U^{\{K\}}\right)$, for the multiplicative Schwarz method.

%%%%%%%%%%%%%%%%%%%%%%%%%%%%%%%%%%%%%%%%%%%%%%%%%%%%%%%%%%%%%%%%%%%%%%%%%%%%%%%%%%%%%%%%%%%%%%%%%%%%
%                                                                                                  %
%   Adjoint problems                                                                               %
%                                                                                                  %
%%%%%%%%%%%%%%%%%%%%%%%%%%%%%%%%%%%%%%%%%%%%%%%%%%%%%%%%%%%%%%%%%%%%%%%%%%%%%%%%%%%%%%%%%%%%%%%%%%%%

\subsubsection{Adjoint problems}
\label{sec:multiplicative_Schwarz_adjoint_problems}

Define solutions $\phi^{[k+i/p]} \in H^{1}_0(\Omega_i)$ of the adjoint problems,
\begin{equation}
\begin{cases}
a_p \left( v,\phi^{[Q + i/p]} \right)
= \tau_{p}^{Q} (v),  \quad \forall v \in H^{1}_0(\Omega_p),\\
a_i\left( v,\phi^{[Q + i/p]} \right)
= \tau_{i}^{Q} (v)
- \sum_{j=i+1}^{p} a_{ij} \left(v, \phi^{[Q + j/p]} \right), \quad 1 \leq i \leq p, \quad \forall v \in H^{1}_0(\Omega_i),
\end{cases}	
\label{eq:multiplicative_adj}
\end{equation}
where
\begin{equation}
\label{eq:def_tau_i_Q}
\tau_{i}^{Q} (v) = \begin{cases}
\displaystyle \sum_{j=1}^p (v, \psi_j)_{i j}, \qquad Q = K-1,\\
\displaystyle -\sum_{j=1}^{i-1} a_{ij} \left( v, \phi^{[Q+1 + j/p]} \right),
\qquad 0 \leq Q < K-1.
\end{cases}	
\end{equation}

The right hand side of \eqref{eq:multiplicative_adj} captures not only the residuals corresponding to the localized QoI data (in the form of $\left(v, \psi_j\right)_{i j}$), but also the \textit{transfer} error between subdomains as the iteration proceeds (in the form of $- \sum_{j=1}^{i-1} a_{ij}\left(v, \phi^{[Q+1 + j/p]}\right) - \sum_{j=i+1}^{p}a_{ij}\left(v, \phi^{[Q + j/p]}\right)$). The adjoint problems \eqref{eq:multiplicative_adj} have the same sequential nature of subdomains solves as the multiplicative Schwarz  Algorithm~\ref{alg:multiplicative_basic}, but note that these are defined backwards from $K, \ K-1+(p-1)/p, \ K-1+(p-2)/p, \cdots, 1$.

\subsubsection{Discretization error}
\label{sec:multiplicative_Schwarz_theorem}

\begin{theorem}[Discretization error for multiplicative Schwarz]
\label{thm:multiplicative_discretization_error}
We have
\begin{equation}
\label{eq:multiplicative_disc_err_rep}
\left(\psi, u^{\{K\}} - U^{\{K\}}\right) = \sum_{k=0}^{K-1} \sum_{i=1}^p  R_i\left(\widetilde{U}^{\{k+i/p\}}, \phi^{[k+i/p]} - \pi_i \phi^{[k+i/p]}\right),
\end{equation}
where $\pi_i v$ is the approximation of $v \in H^1_0(\Omega_i)$ in $V_{i,h,0}$.
\end{theorem}

The proof of Theorem~\ref{thm:multiplicative_discretization_error} is presented as a sequence of lemmas in \S \ref{sec:multiplicative_Schwarz_error_analysis}. The  term $\left( \phi^{[k+i/p]} - \pi_i \phi^{[k+i/p]} \right)$ arises from the use of Galerkin orthogonality, or the fact that the residual of the discrete solution is zero on the finite dimensional space $V_{i,h,0}$. This reflects the fact that $\widetilde{U}^{\{k+i/p\}}$ is the discrete approximation to $u^{\{k+i/p\}}$, not to $u$.

\subsection{\emph{A posteriori} analysis of the discretization error for additive Schwarz}

In this section, we derive representation of the discretization error in the QoI obtained from the additive Schwarz method.

%%%%%%%%%%%%%%%%%%%%%%%%%%%%%%%%%%%%%%%%%%%%%%%%%%%%%%%%%%%%%%%%%%%%%%%%%%%%%%%%%%%%%%%%%%%%%%%%%%%%
%                                                                                                  %
%   Adjoint problems                                                                               %
%                                                                                                  %
%%%%%%%%%%%%%%%%%%%%%%%%%%%%%%%%%%%%%%%%%%%%%%%%%%%%%%%%%%%%%%%%%%%%%%%%%%%%%%%%%%%%%%%%%%%%%%%%%%%%

\subsubsection{Adjoint problems for discretization error}
\label{sec:additive_Schwarz_adjoint_problems}

Define  $\phi^{[k]}_i \in H^{1}_0(\Omega_i)$ solutions to the adjoint problems,
\begin{equation}\label{eq:additive_adjoints}
a_{i}\left(v, \phi_i^{[k]}\right)
= \tau \sum_{j=1}^p \left\{ (\psi_j,v)_{ij} - a_{ij}\left(v, \sum_{l=k+1}^K \phi_j^{[l]} \right) \right \}, \quad \forall v \in H^{1}_0(\Omega_i).
\end{equation}

For a fixed $k$, the adjoint problems   \eqref{eq:additive_adjoints} are independent for each $i$, so $\phi_i^{[k]}$ may be computed backwards from $K, \ K-1, \ K-2, \cdots, 1$ in parallel analogous to the solution strategy in the additive Schwarz Algorithm \ref{alg:additive_basic}. We also note that for implementation purposes $\sum_{l=k+1}^K \phi_j^{[l]}$ involves a sum of the  vectors, $\sum_{l=k+2}^K \phi_j^{[l]}$ (computed earlier) and  $\phi_j^{[k+1]}$.

\subsubsection{Discretization error}
\label{sec:additive_Schwarz_theorem}

%%%%%%%%%%%%%%%%%%%%%%%%%%%%%%%%%%%%%%%%%%%%%%%%%%%%%%%%%%%%%%%%%%%%%%%%%%%%%%%%%%%%%%%%%%%%%%%%%%%%
%                                                                                                  %
%   Theorem: Discretization error representation for additive Schwarz                              %
%                                                                                                  %
%%%%%%%%%%%%%%%%%%%%%%%%%%%%%%%%%%%%%%%%%%%%%%%%%%%%%%%%%%%%%%%%%%%%%%%%%%%%%%%%%%%%%%%%%%%%%%%%%%%%

\begin{theorem}[Discretization error for additive Schwarz]
\label{thm:additive_discretization_error}
We have
\begin{equation}
\label{eq:additive_disc_err_rep}
\left(\psi, e_D^{\{K\}}\right) = \left(\psi, u^{\{K\}} - U^{\{K\}}\right)	 = \sum_{k=1}^{K} \sum_{i=1}^p  R_i\left(\widetilde{U}^{\{k\}}_i, \phi^{[k]}_i - \pi_i \phi^{[k]}_i\right).
\end{equation}
\end{theorem}

The proof of Theorem~\ref{thm:additive_discretization_error} is presented as a sequence of lemmas in \S \ref{sec:additive_Schwarz_error_analysis}.

%%%%%%%%%%%%%%%%%%%%%%%%%%%%%%%%%%%%%%%%%%%%%%%%%%%%%%%%%%%%%%%%%%%%%%%%%%%%%%%%%%%%%%%%%%%%%%%%%%%%
%                                                                                                  %
%   Solution algorithm                                                                             %
%                                                                                                  %
%%%%%%%%%%%%%%%%%%%%%%%%%%%%%%%%%%%%%%%%%%%%%%%%%%%%%%%%%%%%%%%%%%%%%%%%%%%%%%%%%%%%%%%%%%%%%%%%%%%%

\subsection{Solution algorithms}
\label{sec:Schwarz_error_estimation_algorithm}
%\label{sec:multiplicative_Schwarz_error_estimation_algorithm}

The full algorithm for \emph{a posteriori} error estimation for overlapping multiplicative/additive Schwarz domain decomposition is provided in Algorithm \ref{alg:error_estimation}.
%While the overall algorithmic structure for both the multiplicative and additive versions is similar, the algorithms differ in their implementations. For example, the adjoints for the multiplicative solution algorithm are defined by \eqref{eq:multiplicative_adj} and solved in serial, whereas, the adjoints for the additive solution algorithm are defined by \eqref{eq:additive_adjoints} and solved in parallel.

\begin{algorithm}[H]
\caption{Adjoint-based \emph{a posteriori} error estimation for overlapping multiplicative DDMs}
\label{alg:error_estimation}
\begin{algorithmic}

\For{$k=0, 1, 2, \dots, K-1$ }
   \For{$i=1, 2, \dots, p$ }
     \State Solve primal problem on subdomain $i$   \hfill (see \eqref{eq:multiplicative_basic_local}/ \eqref{eq:additive_basic_local})
     \State Combine to construct a global solution  \hfill (see \eqref{eq:multiplicative_basic_global}/ \eqref{eq:additive_basic_global})
  \EndFor
\EndFor

\For{$k=K-1, K-2, \dots, 0$ }%
   \For{$i=p, p-1, \dots, 1$ }
    \State Approximate solution of adjoint problem on subdomain $i$    \hfill (see \eqref{eq:multiplicative_adj}/ \eqref{eq:additive_adjoints})
    \State Compute adjoint weighted residuals and accumulate error contributions \hfill (see \eqref{eq:multiplicative_disc_err_rep}/ \eqref{eq:additive_disc_err_rep})
  \EndFor
\EndFor

\State Approximate solution of  global adjoint problem  \hfill (see \eqref{eq:global_adj})
\State Estimate total error  \hfill (see \eqref{eq:multiplicative_tot_err_rep})
\State Estimate iteration error \hfill (see \eqref{eq:multiplicative_it_err_rep})

\end{algorithmic}
\end{algorithm}

%%%%%%%%%%%%%%%%%%%%%%%%%%%%%%%%%%%%%%%%%%%%%%%%%%%%%%%%%%%%%%%%%%%%%%%%%%%%%%%%%%%%%%%%%%%%%%%%%%%%
%                                                                                                  %
%   Numerical Examples                                                                             %
%                                                                                                  %
%%%%%%%%%%%%%%%%%%%%%%%%%%%%%%%%%%%%%%%%%%%%%%%%%%%%%%%%%%%%%%%%%%%%%%%%%%%%%%%%%%%%%%%%%%%%%%%%%%%%

\section{Numerical examples for multiplicative Schwarz}
\label{sec:mult_Schwarz_numerical_examples}

We provide examples for both multiplicative and additive Schwarz in order to demonstrate the accuracy of the \emph{a posteriori} error estimate for a range of scenarios, stressing the importance of the ability to  distinguish the contributions from discretization and iteration. The initial examples in \S \ref{sec:mult_Schwarz_examples_poisson} are chosen to illustrate certain expected behaviors. We expect the discretization error to decrease as the mesh is  refined, and the iteration error to decrease as we increase the number of iterations or the degree of overlap. We expect the discretization error to be  constant if  the mesh is fixed when the number of subdomains is  increased, but expect the iteration error to increase. In other words, the discretization error is determined by the mesh, but the iteration error is determined by the number (and disposition) of subdomains, the degree of overlap, and the number of iterations. In \S \ref{sec:mult_Schwarz_examples_cancellation}, we construct a problem where the discretization and iteration errors have opposite signs, and show that iterating with a fixed mesh may result in the overall error initially decreasing as the iteration error decreases, achieving a minimum when the discretization and iteration errors cancel each other, and then increasing (and stabilizing) as the discretization error comes to dominate the total error. For the convection dominated problem in \S \ref{sec:mult_Schwarz_examples_convection_diffusion}, we show how the configuration of the subdomains affects the iteration error, but not the discretization error. Finally in \S \ref{sec:mult_Schwarz_example_two_stage_strategy} we provide two examples of a two stage strategy in which an accurate error estimate for an initial coarse solution guides the construction of a more accurate ``production'' calculation. We choose to locally adapt the finite element mesh when the discretization error in a particular subdomain is dominant, and to increase the degree of subdomain overlap when iteration is the leading source of error. ``Adaptivity'' in the context of iterative methods, requires strategies for addressing both discretization and iteration errors.

%----------------------------------------------------------------------------------------------------
\subsection{Error estimates and effectivity ratios}
\label{sec:mult_Schwarz_error_estimates_effectivity}

We compute approximate adjoint solutions $\Phi^{[k+i/p]} \approx \phi^{[k+i/p]}$, and $\Phi \approx \phi$
and then compute \eqref{eq:multiplicative_disc_err_rep} and \eqref{eq:multiplicative_tot_err_rep}.
The resulting  error estimates are
\begin{equation}
\label{eq:multiplicative_disc_err_est}
\eta_D^K \equiv \sum_{k=0}^{K-1} \sum_{i=1}^p  R_i(\widetilde{U}^{\{k+i/p\}}, \Phi^{[k+i/p]} - \pi_i \Phi^{[k+i/p]}),
\end{equation}
and
\begin{equation}
\label{eq:multiplicative_tot_err_est}
\eta^K \equiv  R (U^{\{K\}}, \Phi).
\end{equation}

One way to measure the  performance of an error estimates is the ``effectivity ratios'',
\begin{equation}
	\gamma= \frac{\eta^K}{(u - U^{\{K\}}, \psi)} \,,
\end{equation}
and
\begin{equation}
	\gamma_D = \frac{\eta_D^K}{(u^{\{K\}} - U^{\{K\}}, \psi)} \,.
\end{equation}
An effectivity ratio close to one indicates that the error estimate is accurate. We also recall that $e_I$ denotes the iteration error.

%----------------------------------------------------------------------------------------------------
\subsection{Poisson's equation}
\label{sec:mult_Schwarz_examples_poisson}

Consider  Poisson's equation
\begin{equation}
\label{eq:Poissons_equation}
\begin{aligned}
	-\nabla^2 u  &= f, \quad \text{ in } \Omega, \\
		       u &= 0, \quad \text{ on } \partial \Omega,
\end{aligned}
\end{equation}
in a square domain $\Omega = [0 , 1] \times [0 , 1]$, where $f(x,y) = 8\pi^2\sin(2 \pi x) \sin(2 \pi y)$. The QoI in \eqref{eq:QoI_def} is specified by
\begin{equation}
\label{eq:QoI_4_Poisson}
\psi = \mathbbm{1}_{[.6, \, .8]\times [.6, \, .8]}.
\end{equation}
where $\mathbbm{1}_\omega$ is the characteristic function on a domain $\omega$.
In the computations below, unless otherwise specified, the mesh  is uniform and contains $2 \times N_x \times N_y$ triangular elements. The overlap between subdomains is indicated by $\beta$.
%%% TODO: Is it clear what beta means?

%Cancellation of error example has $\psi = \chi_{[0.4, \, 0.8]\times [0.4, \, 0.8]}, \beta = 0.05, N_x = N_y = 40$.

\subsubsection{$2 \times 1$ subdomains}
\label{sec:mult_Schwarz_example_poisson_2x1}

Two overlapping subdomains $\Omega_1 = [0, .6] \times [0, 1]$ and $\Omega_2 = [.4, 1] \times [0, 1]$ are illustrated in Figure~\ref{fig:2x1}, corresponding to an overlap parameter $\beta = 0.1$. The solid black lines in this figure and in subsequent figures, indicate the center line between overlapping subdomains.

\begin{figure}[!ht]
\centering
\subfloat[]{
\includegraphics[width=0.3\textwidth] {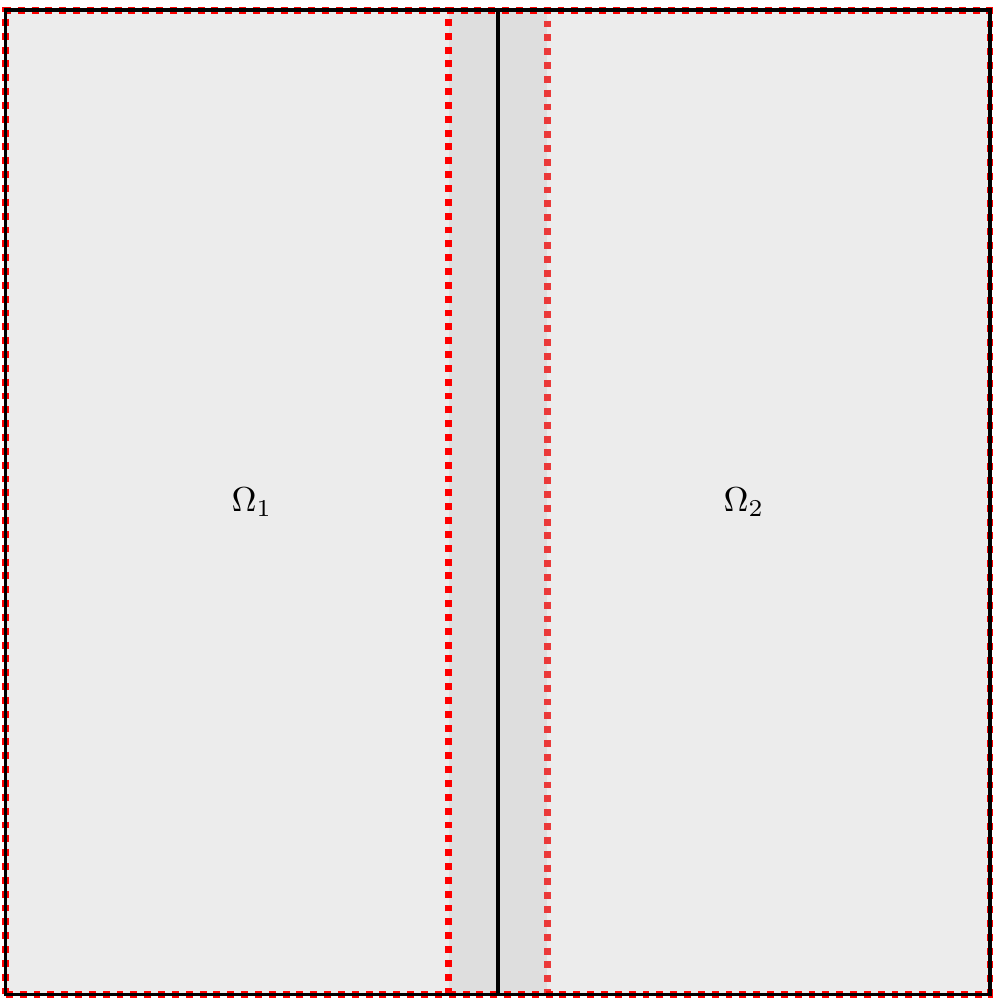}
\label{fig:2x1}
}
 \hfill
\subfloat[]{
\includegraphics[width=0.3\textwidth] {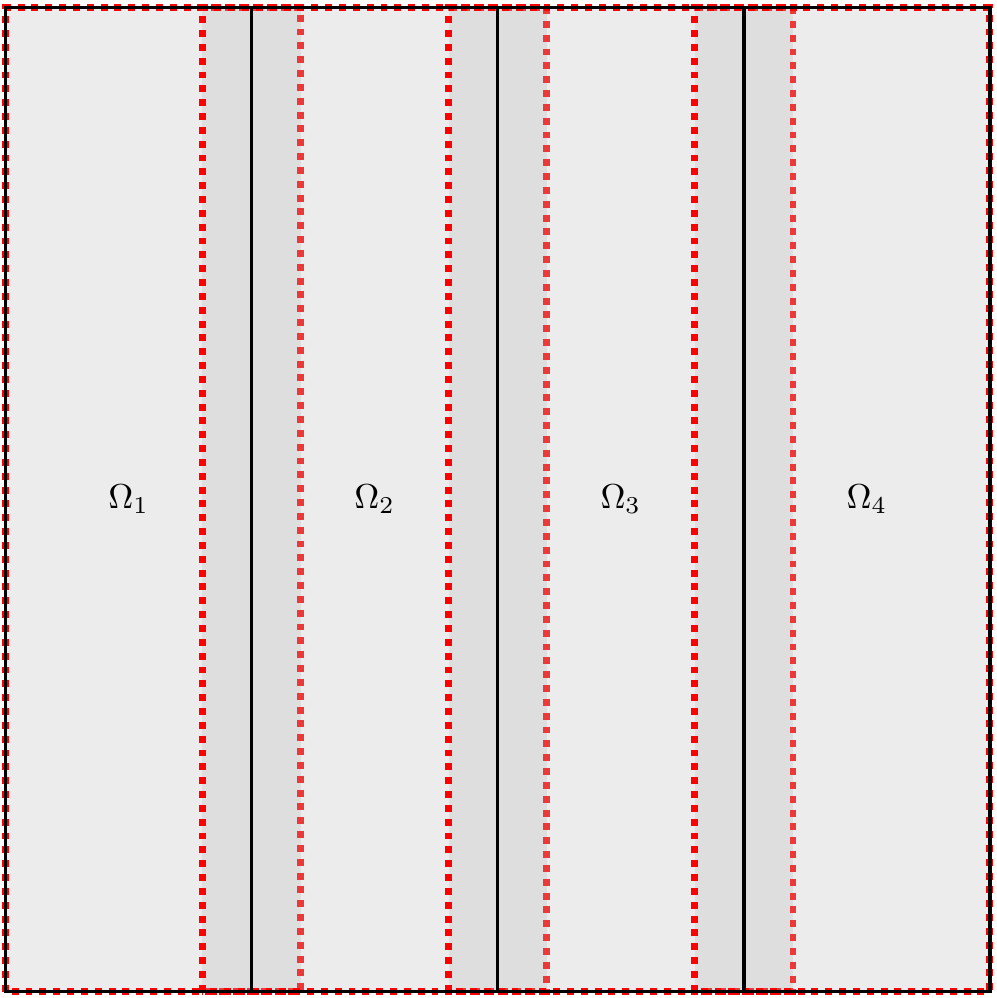}
\label{fig:4x1}
}
 \hfill
\subfloat[]{
\includegraphics[width=0.3\textwidth] {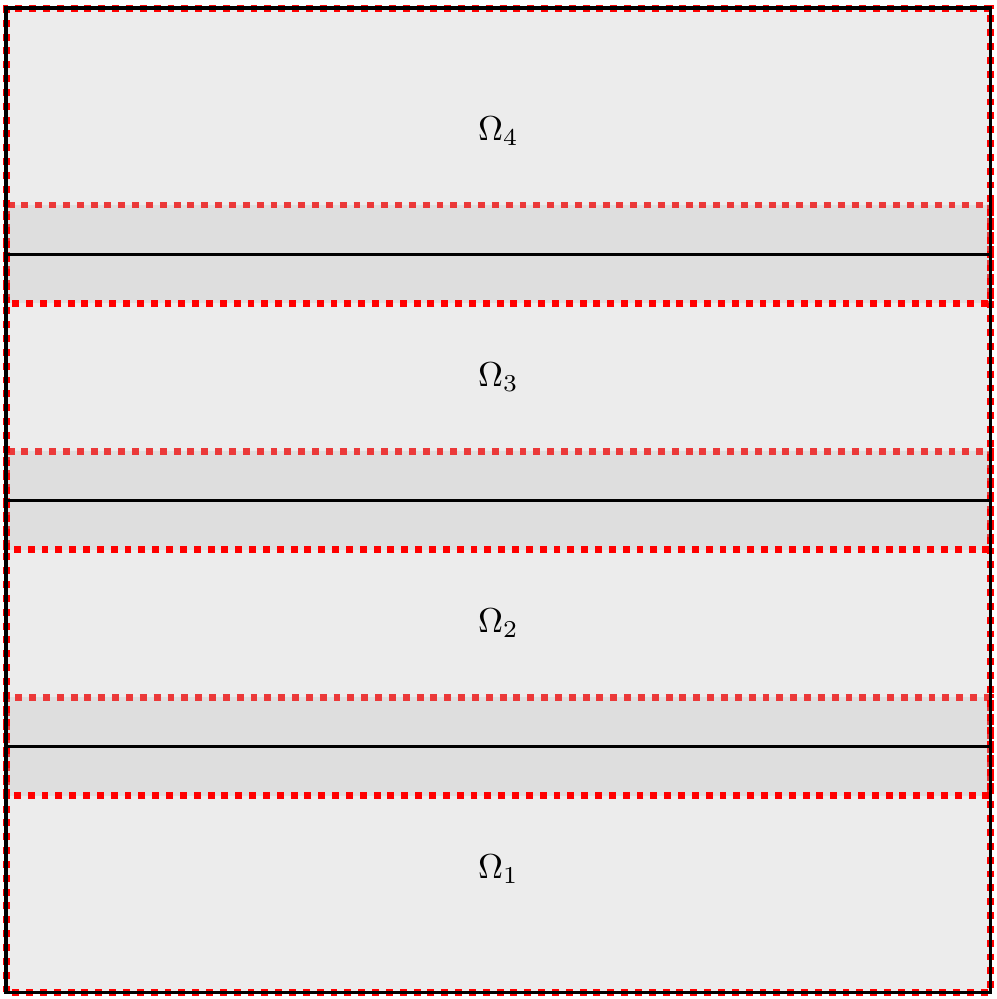}
\label{fig:1x4}
}
\caption{Overlapping subdomains with $\beta = 0.1$. (a) Two ($2 \times 1$) subdomains. (b) Four ($4 \times 1$)  subdomains. (c) Four ($1 \times 4$)  subdomains.}
\label{fig:beta_0_1_figs}
\end{figure}

Estimates of the discretization, iteration and total errors, and the corresponding effectivity ratios as we vary the overlap $\beta$, number of Schwarz iterations $K$ and number of elements are shown in Tables~\ref{tab:mult_Schwarz_example_poisson_2x1}. In all cases the effectivity ratios are close to 1.0. The table highlights the sensitivity of the estimates to the different contributions of the error. The ``base'' computation with $N_x=N_y=20, \beta=0.1$ and $K=2$ is repeated for ease of comparison. Increasing the overlap decreases the iteration error $e_I^{\{K\}}$ but does not have a significant effect on the discretization error $e_D^{\{K\}}$. The iteration error decreases with increasing number of Schwarz iterations, but the discretization error is largely unaffected when the mesh is fixed.  The discretization error decreases when the mesh is refined with a fixed number of iterations, but the iteration error remains essentially constant.

\begin{table}[H]
\centering
\begin{tabular}{||c|c|c|c|c|c|c|c|c||}
\hline
$N_x$ & $N_y$ & $\beta$ & $K$ & Est. Err. & $\gamma$ &  $e_D^{\{K\}}$ & $\gamma_D$  &    $e_I^{\{K\}}$  \\
%\hline
%20 & 20 & 0.1 & 2 &	1.02e-03 & 0.998 & 6.56e-04 & 0.998 & 3.60e-04	\\
%20 & 20 & 0.2 &	2 & 7.03e-04 & 0.996 & 6.28e-04 & 0.997 & 7.50e-05	\\
%\hline
%20 & 20 & 0.1 & 2 &	1.02e-03 & 0.998 & 6.56e-04 & 0.998 & 3.60e-04	\\
%20 & 20 & 0.1 & 4 &	6.55e-04 & 0.996 & 6.26e-04 & 0.997 & 2.89e-05	\\
%\hline
%20 & 20 & 0.1 & 2 & 1.02e-03 & 0.998 & 6.56e-04 & 0.998 & 3.60e-04	\\
%40 & 40 & 0.1 & 2 & 5.25e-04 & 1.00  & 1.66e-04 & 0.999 & 3.60e-04	\\
%\hline
\hline
20 & 20 & 0.1 & 2 &	1.02e-03 & 9.98E-01 & 6.56e-04 & 9.98E-01 & 3.60e-04	\\
20 & 20 & 0.2 &	2 & 7.03e-04 & 9.96E-01 & 6.28e-04 & 9.97E-01 & 7.50e-05	\\
\hline
20 & 20 & 0.1 & 2 &	1.02e-03 & 9.98E-01 & 6.56e-04 & 9.98E-01 & 3.60e-04	\\
20 & 20 & 0.1 & 4 &	6.55e-04 & 9.96E-01 & 6.26e-04 & 9.97E-01 & 2.89e-05	\\
\hline
20 & 20 & 0.1 & 2 & 1.02e-03 & 9.98E-01 & 6.56e-04 & 9.98E-01 & 3.60e-04	\\
40 & 40 & 0.1 & 2 & 5.25e-04 & 1.00E+00 & 1.66e-04 & 9.99E-01 & 3.60e-04	\\
\hline
\end{tabular}
\caption{Multiplicative Schwarz for Poisson's equation: $2 \times 1$ subdomains.}
\label{tab:mult_Schwarz_example_poisson_2x1}
\end{table}

\subsubsection{$4 \times 1$ subdomains}
\label{sec:mult_Schwarz_example_poisson_4x1}

The computational domains for $\beta=0.1$ are shown in Figure~\ref{fig:4x1}.  It is well known that as the number of subdomains increases, the convergence of Schwarz methods decreases, and this is evident by comparing Tables~\ref{tab:mult_Schwarz_example_poisson_4x1} and \ref{tab:mult_Schwarz_example_poisson_2x1}. While the discretization errors of the four subdomain and two subdomain cases are comparable in magnitude, the iteration error $e_I^{\{K\}}$ is an order of magnitude larger for four subdomains compared to two.  The contributions of the separate components of the total error vary with the overlap, number of iterations and number of elements in a qualitatively similar way to the results  in \S~\ref{sec:mult_Schwarz_example_poisson_2x1}.

\begin{table}[H]
\centering
\begin{tabular}{||c|c|c|c|c|c|c|c|c||}
\hline
$N_x$ & $N_y$ & $\beta$ & $K$ & Est. Err. & $\gamma$ &  $e_D^{\{K\}}$ & $\gamma_D$  &    $e_I^{\{K\}}$  \\
\hline
20 & 20 & 0.1 & 2 & 4.57e-03 &	9.99e-01 &	6.92e-04 &	9.97e-01 &	3.88e-03	\\
20 & 20 & 0.2 & 2 &	1.34e-03 &	9.98e-01 &	6.48e-04 &	9.98e-01 &	6.87e-04	\\
\hline
20 & 20 & 0.1 & 2 &	4.57e-03 &	9.99e-01 &	6.92e-04 &	9.97e-01 &	3.88e-03	\\
20 & 20 & 0.1 & 4 &	1.04e-03 &	9.98e-01 &	6.48e-04 &	9.98e-01 &	3.94e-04	\\
\hline
20 & 20 & 0.1 & 2 &	4.57e-03&	9.99e-01&	6.92e-04&	9.97e-01&	3.88e-03	\\
40 & 40 & 0.1 & 2 &	4.05e-03&	1.00e+00&	1.75e-04&	9.99e-01&	3.88e-03	\\
\hline
\end{tabular}
\caption{Multiplicative Schwarz for Poisson's equation: $4 \times 1$ subdomains.}
\label{tab:mult_Schwarz_example_poisson_4x1}
\end{table}

\subsubsection{$4 \times 4$ subdomains}
\label{sec:mult_Schwarz_example_poisson_4x4}

The computational domains for $\beta=0.1$ and sixteen equally-sized subdomains are configured in a $4 \times 4$ grid, see Figure~\ref{fig:4x4}. The error estimates are quite accurate. The results, shown in Table \ref{tab:mult_Schwarz_example_poisson_4x4} are qualitatively similar to those in Tables \ref{tab:mult_Schwarz_example_poisson_2x1} and \ref{tab:mult_Schwarz_example_poisson_4x1}. The iteration error is larger than in  the $4 \times 1$ case, while the discretization errors are essentially the same in both, which is to be expected since the finite element meshes are the same.

\begin{table}[H]
\centering
\begin{tabular}{||c|c|c|c|c|c|c|c|c||}
\hline
$N_x$ & $N_y$ & $\beta$ & $K$ & Est. Err. & $\gamma$ &  $e_D^{\{K\}}$ & $\gamma_D$  &    $e_I^{\{K\}}$  \\
\hline
20 & 20 & 0.1 & 2 &	9.22e-03 &	1.00e+00 &	1.02e-03 &	1.00e+00 &	8.20e-03	\\
20 & 20 & 0.2 &	2 & 2.80e-03 &	9.99e-01 &	7.31e-04 &	9.98e-01 &	2.07e-03	\\
\hline
20 & 20 & 0.1 & 2 &	9.22e-03 &	1.00e+00 &	1.02e-03 &	1.00e+00 &	8.20e-03	\\
20 & 20 & 0.1 & 4 &	2.72e-03 &	9.99e-01 &	8.27e-04 &	9.99e-01 &	1.90e-03	\\
\hline
20 & 20 & 0.1 & 2 &	9.22e-03&	1.00e+00&	1.02e-03&	1.00e+00&	8.20e-03	\\
40 & 40 & 0.1 & 2 &	8.45e-03&	1.00e+00&	2.55e-04&	1.00e+00&	8.20e-03	\\
\hline
\end{tabular}
\caption{Multiplicative Schwarz for Poisson's equation: $4 \times 4$ subdomains.}
\label{tab:mult_Schwarz_example_poisson_4x4}
\end{table}

%----------------------------------------------------------------------------------------------------
\subsection{Cancellation of error}
\label{sec:mult_Schwarz_examples_cancellation}

To illustrate the potential for cancellation between discretization and iteration errors, the quantity of interest is chosen to be
\begin{equation}
\psi = \mathbbm{1}_{[.4, \, .8]\times [.4, \, .8]}.
\end{equation}
for two subdomains and an overlap $\beta=0.05$. Computational results for an increasing number of Schwarz iterations are shown in Table~\ref{tab:mult_Schwarz_examples_cancellation}. The magnitude of the total error initially decreases as the iteration proceeds, reaching a minimum after six iterations, but then starts to increase. This behavior is explained by observing that the discretization and iteration errors have opposite signs. The discretization error is essentially fixed as the iteration proceeds and has a value of $-1.6 \times 10^{-4}$. The initial iteration error is of order $4.0 \times 10^{-3}$ and dominates the total error. As expected, the iteration error decreases monotonically as $K$ increases, but is always positive. After six iterations the discretization and iteration errors have approximately equal magnitudes but opposite signs, and cancel to produce a total error of $3.0\times 10^{-5}$. For greater than six iterations, the iteration error continues to decrease and now the discretization error dominates the total error. The total error increases to $-1.5\times 10^{-4}$ after 10 iterations and gradually approaches the (fixed) discretization error as the number of iterations increases further.

\begin{table}[H]
\centering
\begin{tabular}{||c|c|c|c|c|c||}
\hline
$K$ & Est. Err. & $\gamma$ &  $e_D^{\{K\}}$ & $\gamma_D$  & $e_I^{\{K\}}$  \\
\hline
1&	3.98e-03&	1.00e+00&	-1.50e-05&	9.98e-01&	4.00e-03	\\
2&	2.07e-03&	1.00e+00&	-5.95e-05&	9.99e-01&	2.13e-03	\\
3&	1.04e-03&	1.00e+00&	-9.11e-05&	1.00e+00&	1.13e-03	\\
4&	4.89e-04&	1.00e+00&	-1.14e-04&	1.00e+00&	6.03e-04	\\
5&	1.91e-04&	1.00e+00&	-1.30e-04&	1.00e+00&	3.21e-04	\\
6&	2.97e-05&	1.01e+00&	-1.41e-04&	1.00e+00&	1.71e-04	\\
7&	-5.83e-05&	9.96e-01&	-1.49e-04&	1.00e+00&	9.09e-05	\\
8&	-1.07e-04&	9.98e-01&	-1.55e-04&	1.00e+00&	4.84e-05	\\
9&	-1.33e-04&	9.98e-01&	-1.59e-04&	1.00e+00&	2.58e-05	\\
10&	-1.48e-04&	9.98e-01&	-1.62e-04&	1.00e+00&	1.37e-05	\\
\hline
\end{tabular}
\caption{Multiplicative Schwarz for Poisson's equation: $2 \times 1$ subdomains, $N_x = N_y = 40, \ \beta = 0.05$.}
\label{tab:mult_Schwarz_examples_cancellation}
\end{table}

%----------------------------------------------------------------------------------------------------
\subsection{A convection-diffusion problem}
\label{sec:mult_Schwarz_examples_convection_diffusion}

Consider the convection-diffusion equation,
\begin{equation}
\label{eq:convection_diffusion}
\begin{aligned}
	- \nabla^2 u  + \mathbf{b}\cdot \nabla u =& f, \quad \text{ in } \Omega, \\
		u =& 0, \quad \text{ on } \partial \Omega,
\end{aligned}
\end{equation}
where $\Omega = [0 , 1] \times [0 , 1] , f(x,y) = 1$, and $\mathbf{b} = [-60,0]$. The effect of the convection  is that a perturbation to data on the right affects the solution to the left. For this example, we choose the quantity of interest
\begin{equation}
\label{eq:convection_diffusion_QoI}
\psi = \mathbbm{1}_{[.05, \, .2]\times [.05, \, .2]},
\end{equation}
concentrated near the bottom left hand corner. %This $\psi$ is chosen as there are sharp gradients in the solution near the left boundary due to the convective vector field $\mathbf{b}$.
The adjoint problems are solved using continuous piecewise cubic polynomials to ensure accurate solutions in the presence of the strong convective vector field. We experiment with two configurations with the subdomains aligned with different coordinate axes, and either parallel with or perpendicular to the direction of convection.

\subsubsection{$4 \times 1$ configuration}

This subdomain configuration is the same as in Figure~\ref{fig:4x1}. The total, discretization and iteration errors are provided in Table~\ref{tab:mult_Schwarz_examples_convection_diffusion}. Note the significant iteration error in this configuration for $K=2$, which dominates the total error. The large iteration error for $K=2$ is to be expected given direction of information travel from right to left. The iteration error decreases dramatically for $K=4$ and $K=6$, and discretization error becomes the dominant error.

\subsubsection{$1 \times 4$ configuration}
This subdomain configuration is shown in Figure~\ref{fig:1x4}. The subdomains are aligned with the direction of the convective vector field. The iteration error after two iterations and the total error are more than an order of magnitude less than in the $4 \times 1$ case. In this scenario, one subdomain contains most of the ``domain of influence'' for the QoI~\cite{EHL2005} and hence results in low iteration error, even for $K=2$. There is again cancellation between the discretization and iteration errors for $K=2$ so that the total error increases for $K=4$ and $K=6$ with the total error dominated by the discretization error.

\begin{figure}[!ht]
\centering
\subfloat[]{
\includegraphics[width=0.3\textwidth] {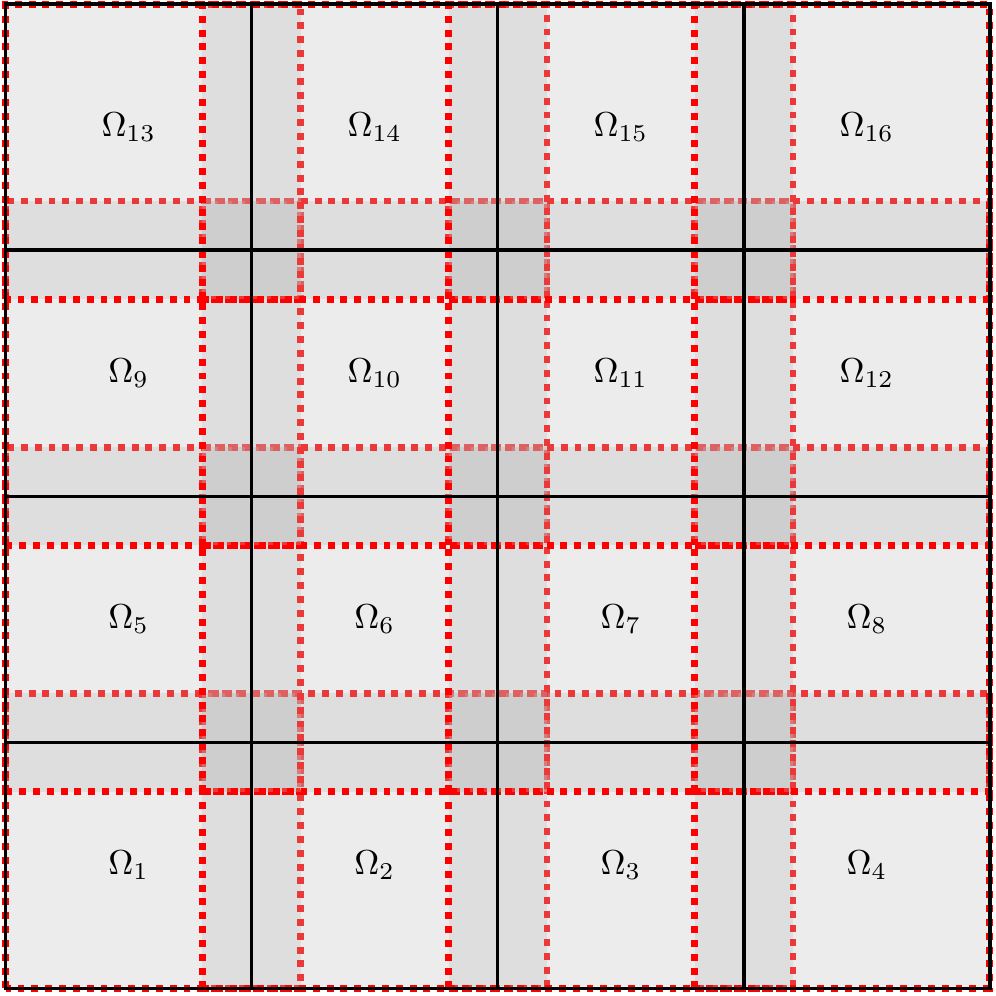}
\label{fig:4x4}
}
% \hspace{1in}
\hfill
\subfloat[]{
\includegraphics[width=0.3\textwidth] {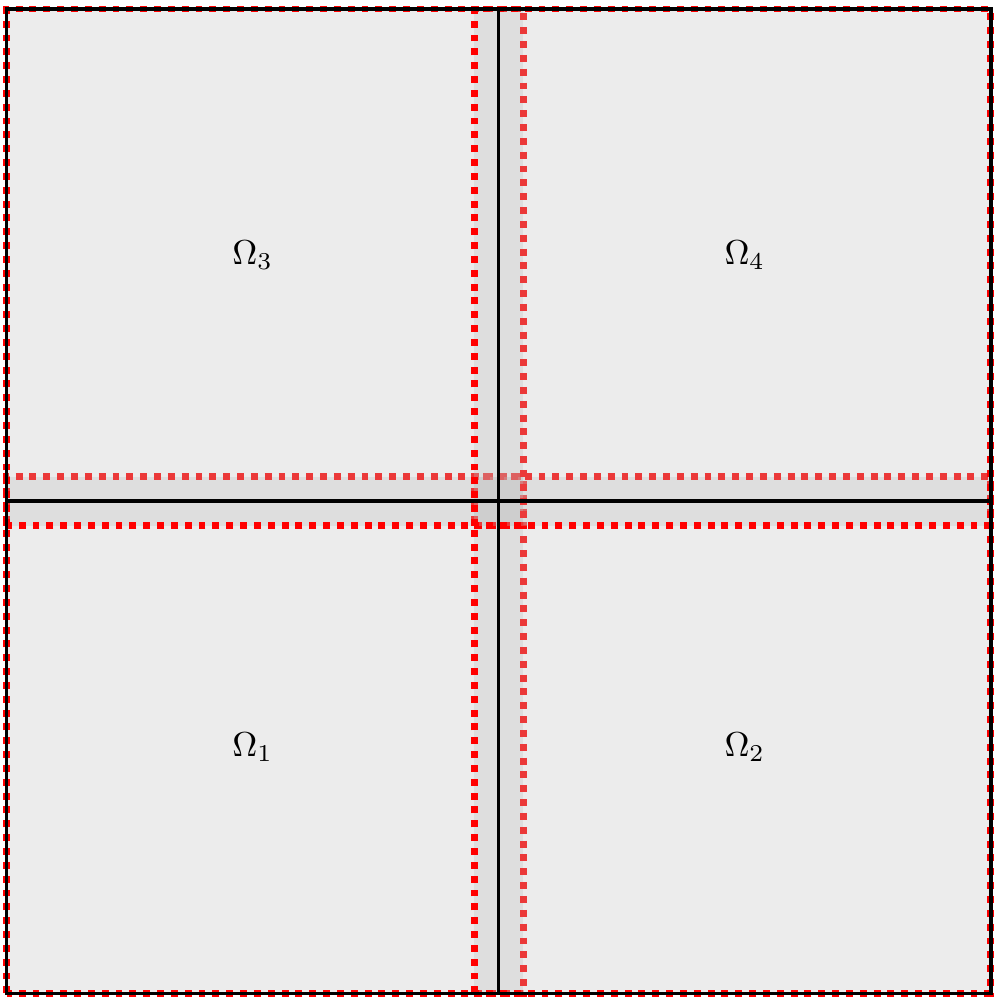}
\label{fig:2x2_p05}
}
\hfill
\subfloat[]{
\includegraphics[width=0.3\textwidth] {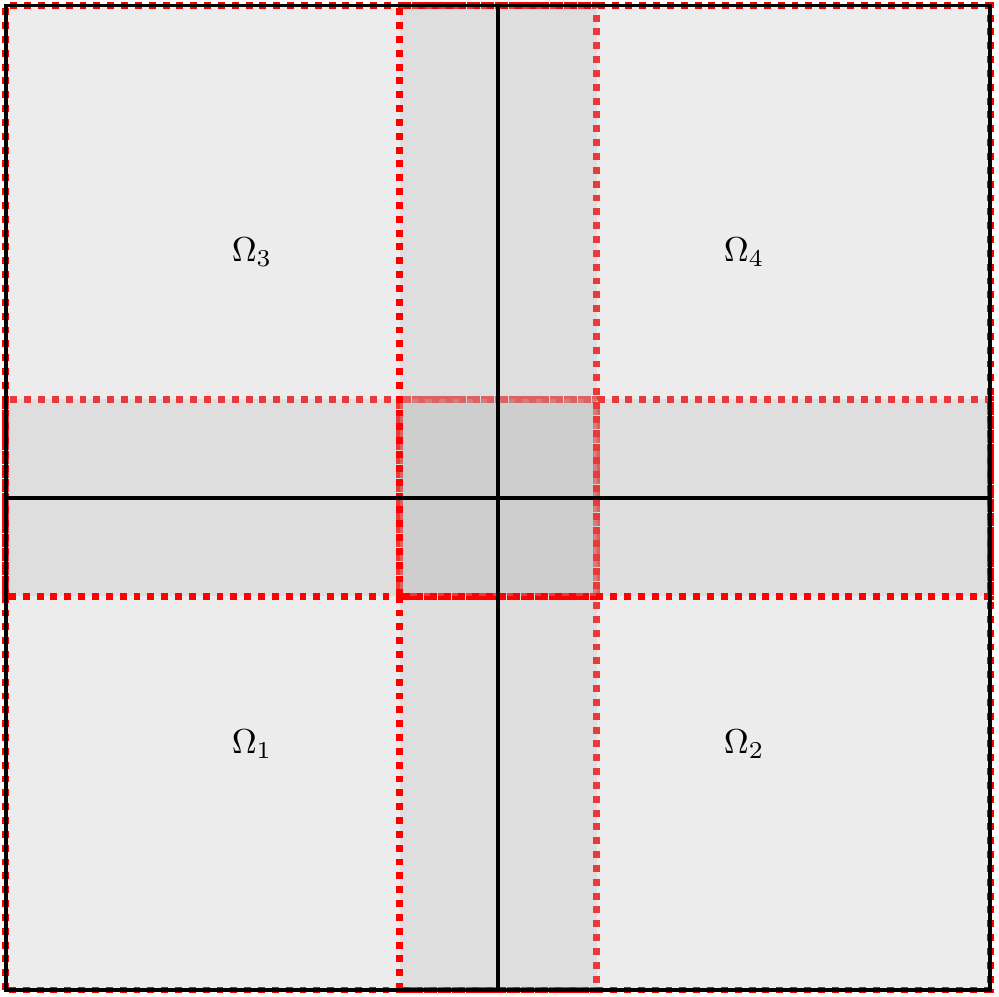}
\label{fig:2x2_p2}
}
\caption{ (a) Sixteen  ($4 \times 4$) overlapping subdomains  with $\beta = 0.1$.
(b) Four ($2 \times 2$) overlapping subdomains with $\beta = 0.05$.
(c) Four ($2 \times 2$) overlapping subdomains with $\beta = 0.2$.
}
\label{fig:mixed_beta_figs}
\end{figure}

\begin{table}[H]
\centering
\begin{tabular}{||c|c|c|c|c|c||}
\hline
$K$ & Est. Err. & $\gamma$ &  $e_D^{\{K\}}$ & $\gamma_D$  & $e_I^{\{K\}}$  \\
\hline
\multicolumn{6}{||c||}{$4 \times 1$ configuration} \\
\hline
2&	9.76e-03&	1.00e+00&	-1.54e-04&	9.87e-01&	9.92e-03	\\
4&	-1.15e-04&	9.81e-01&	-1.42e-04&	9.77e-01&	2.67e-05	\\
6&	-3.54e-04&	9.94e-01&	-3.54e-04&	9.94e-01&	4.36e-10	\\
\hline
\multicolumn{6}{||c||}{$1 \times 4$ configuration} \\
\hline
2&	8.42e-05&	1.03e+00&	-3.73e-04&	9.94e-01&	4.57e-04	\\
4&	-3.54e-04&	9.94e-01&	-3.55e-04&	9.94e-01&	3.53e-07	\\
6&	-3.54e-04&	9.94e-01&	-3.54e-04&	9.94e-01&	3.70e-11	\\
\hline
\end{tabular}
\caption{Multiplicative Schwarz for convection-diffusion: $N_x = N_y = 20, \ \beta =  0.1$.}
\label{tab:mult_Schwarz_examples_convection_diffusion}
\end{table}

%----------------------------------------------------------------------------------------------------
\subsection{Two stage solution strategy for Poisson's equation}
\label{sec:mult_Schwarz_example_two_stage_strategy}

Adjoint-based \emph{a posteriori} error estimates can provide useful information for designing efficient two stage strategies for computing approximate solutions. First, a preliminary, inexpensive computation is performed on a coarse discretization. The \emph{a posteriori} error estimate for the  ``stage 1'' solution is computed and the different error contributions determined. A  more expensive ``stage 2'' approximation is computed  using numerical parameters chosen to balance the sources of error. We provide two examples of this strategy below. The stage 1 computation for both experiments is run on a $2 \times 2$ subdomain configuration as shown in Figure~\ref{fig:2x2_p2}.

\subsubsection{Dominant discretization error }
\label{sec:mult_Schwarz_num_exp_disc_err_dom}

Consider the QoI given by \eqref{eq:QoI_4_Poisson}. The results on the initial $2 \times 2$ subdomain configuration with $N_x = N_y = 10, \beta = 0.2$ and  $K = 6$ are provided in  Table~\ref{tab:mult_Schwarz_example_two_stage_strategy_discretization}. The mesh for this computation is shown in Figure~\ref{fig:2x2_mesh_uniform}. The main source of the error is the discretization error $e_D^{\{K\}}$. In order to reduce the discretization error, we need to reduce the discretization error contribution arising from each subdomain. We define the contribution to the discretization error from subdomain $i$ as
\begin{equation}
	S_i^K = \sum_{k=0}^{K-1} R_i( \widetilde{U}^{\{k+i/p\}}, \phi^{[k+i/p]} - \pi_i \phi^{[k+i/p]} ), \qquad i=1, \dots, p,
\end{equation}
so that the discretization error, \eqref{eq:multiplicative_disc_err_rep} may be written as
\begin{equation}
(\psi, u^{\{K\}} - U^{\{K\}})	 =  \sum_{i=1}^p  S_i^K.
\end{equation}
The values of $S_i^K$ for the stage 1 calculation are also shown in Table~\ref{tab:mult_Schwarz_example_two_stage_strategy_discretization}. Subdomain 4 contributes the most towards the discretization error, and hence it is the prime candidate for refinement.  After refining all the elements in subdomain 4, the refined mesh is shown in Figure~\ref{fig:2x2_mesh_ref}. The discretization errors in each subdomain $S_i^K$ and the total error after the refinement are shown in Table~\ref{tab:mult_Schwarz_example_two_stage_strategy_discretization}. The discretization error is significantly lower and hence the total error is also significantly lower. The values of $S_i^K$ also indicate that now each subdomain contributes roughly the same magnitude towards the discretization error.

We note that we can take advantage of  cancellation of the discretization errors. Applying the standard approximation theory for degree one Lagrange finite elements, we expect the discretization error component $S_4^K$ to decrease by a factor of four if we refine the mesh corresponding to subdomain 4 uniformly. The conjectured value for $S_4^K$ is therefore approximately $9 \times 10^{-4}$.  The discretization errors from subdomains 2 and 3 (represented by $S_2^K$ and $S_3^K$) have negative signs and are not expected to change as significantly when subdomain 4 is refined. As shown in Table~\ref{tab:mult_Schwarz_example_two_stage_strategy_discretization}, after refinement of subdomain 4, there is significant cancellation of error between subdomains 2 and 3 and subdomain 4, and the total error is  $3.44 \times 10^{-4}$. \emph{Uniformly} refining the entire initial mesh results in a refined mesh with 441 vertices (shown in Figure~\ref{fig:2x2_mesh_uniform_ref}) and the solution after $K=6$ iterations has a total error of $6.24 \times 10^{-4}$, which is approximately the expected four-fold reduction in error. The mesh refined using adjoint based error information in Figure~\ref{fig:2x2_mesh_ref} has almost half the number of degree of freedoms of the uniformly refined mesh in Figure~\ref{fig:2x2_mesh_uniform_ref}, but none-the-less has half the total error ($3.44 \times 10^{-4}$ vs. $6.24 \times 10^{-4}$). Recognizing and taking advantage of cancellation of error can produce otherwise startling efficiencies. Similar refinement strategies, where specific ``components'' are refined to exploit cancellation of error, are also employed in \cite{Chaudhry17, chaudhry2016posteriori}. Such component-wise refinement strategies allow for estimation of the decrease of error in a more reliable manner than classical adaptive refinement strategies in which disparate elements are marked for refinement.

%This uniformly mesh has almost twice the number of degree of freedoms than the  mesh refined using adjoint based error information (shown in Figure~\ref{fig:2x2_mesh_ref})  but still has twice the error (as may be verified by viewing Table~\ref{tab:mult_Schwarz_example_two_stage_strategy_discretization}),

\begin{figure}[!ht]
\centering
\subfloat[]{
\includegraphics[width=0.3\textwidth] {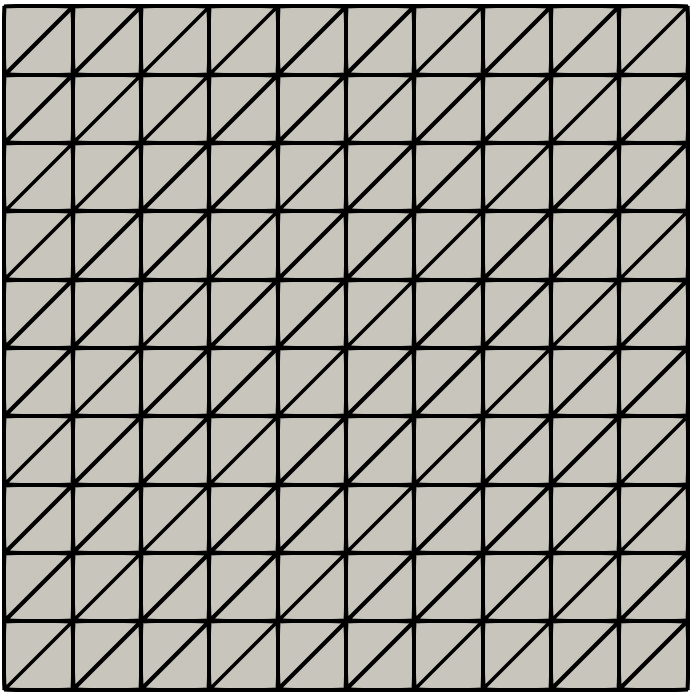}
\label{fig:2x2_mesh_uniform}
}
% \hspace{1in}
\hfill
\centering
\subfloat[]{
\includegraphics[width=0.3\textwidth] {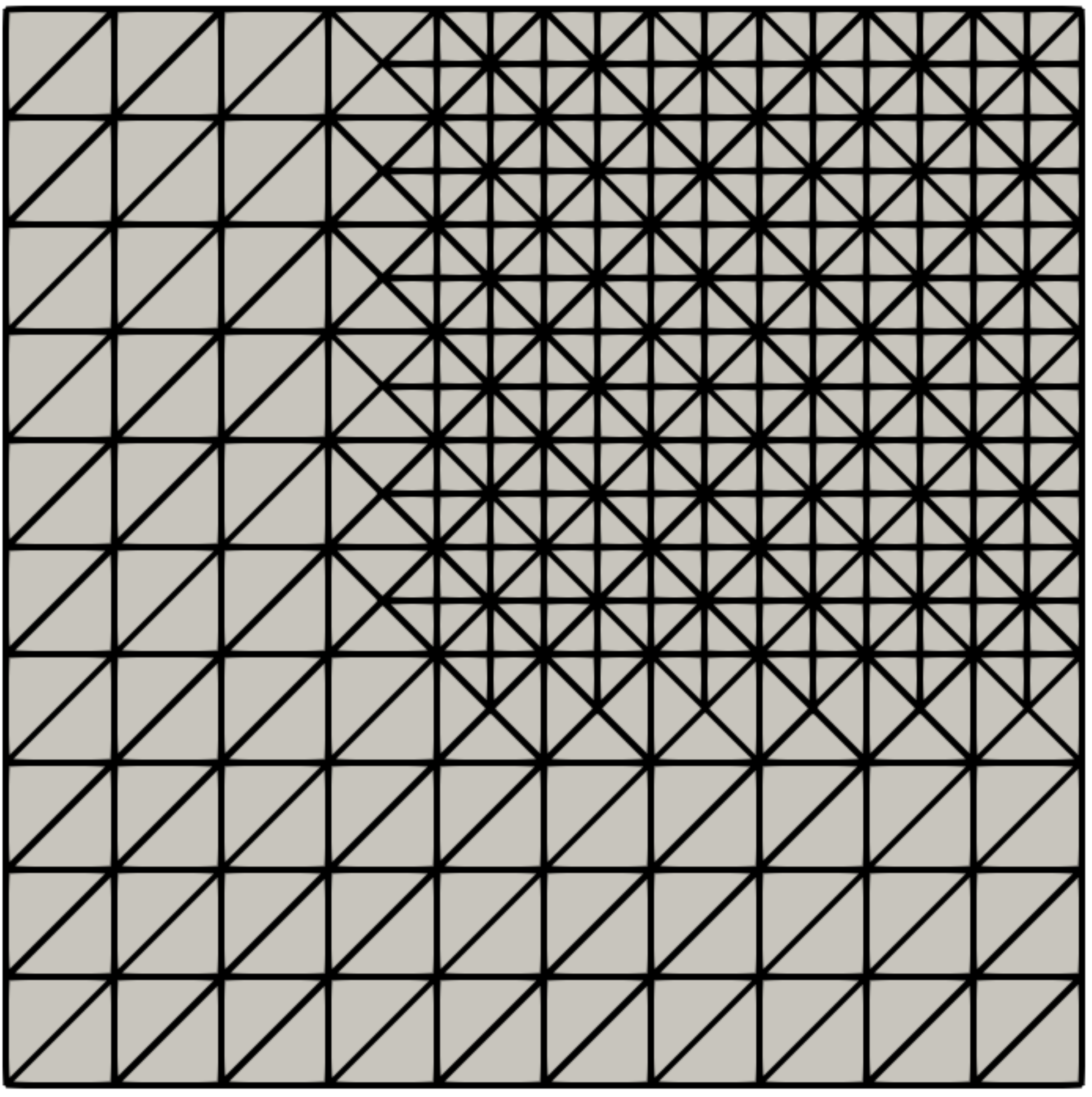}
\label{fig:2x2_mesh_ref}
}
\hfill
\centering
\subfloat[]{
\includegraphics[width=0.3\textwidth] {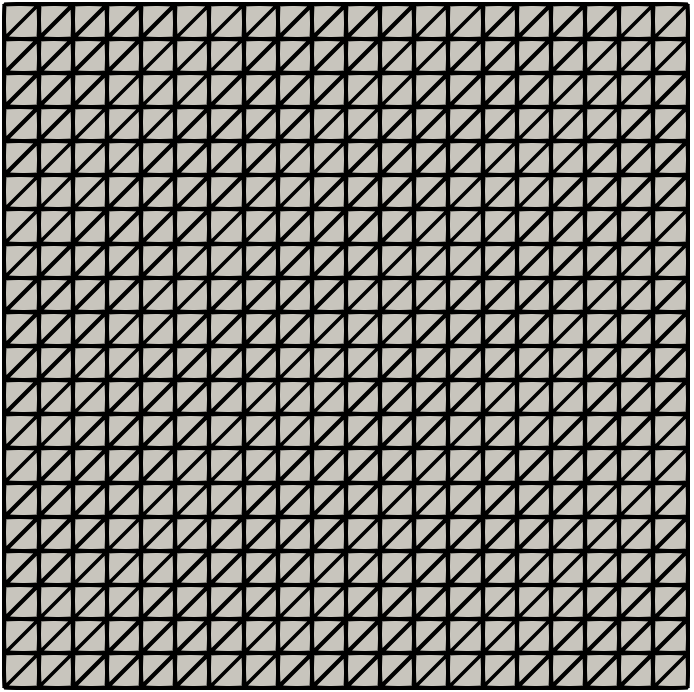}
\label{fig:2x2_mesh_uniform_ref}
}
\caption{ (a) Initial uniform mesh. (b) Mesh refinement  in $\Omega_4$ guided by adjoint based error estimates. (c) Uniformly refined mesh.
}
\label{fig:mesh_ref_2x2}
\end{figure}

\begin{table}[H]
\centering
\begin{tabular}{||c|c|c|c|c|c|c|| }
\hline
Stage & Num. vertices  & Est. Err. & $\gamma$ &  $e_D^{\{K\}}$ & $\gamma_D$  & $e_I^{\{K\}}$  \\
\hline
1 & 121 & 2.36e-03 & 9.83e-01 &	2.36e-03 &	9.89e-01 &	6.98e-06  \\
2 & 253 & 3.44e-04 & 1.00e+00 &	3.37e-04 &	1.05e+00 &	6.97e-06  \\
\hline
Stage & Num. vertices &  $i$ 	& 1 	& 2 	& 3 	& 4   \\
\hline
1 & 121 & $S_i^K$ & 3.07e-04 & -7.94e-04 & -7.82e-04 & 3.62e-03 \\
2 & 253 & $S_i^K$ & 1.82e-04 & -3.87e-04 & -3.85e-04 & 9.27e-04 \\
\hline
\end{tabular}
\caption{Two stage solution strategy using multiplicative Schwarz to solve Poisson's equation: $\beta = 0.2, \ K=6$.}
\label{tab:mult_Schwarz_example_two_stage_strategy_discretization}
\end{table}

\subsubsection{Dominant iteration error}
\label{sec:mult_Schwarz_num_exp_it_err_dom}

For the same choice of QoI, we perform a stage 1 computation with $2 \times 2$ subdomains and $N_x = N_y = 40, \beta = 0.05$ and $K=2$. This configuration is shown in Figure \ref{fig:2x2_p05}. The contributions to the total error are shown in Table~\ref{tab:mult_Schwarz_example_two_stage_strategy_iteration}. The dominant source of the error is the iteration error $e_I^{\{K\}}$. There are two ways to reduce it, either by performing a great number of iterations or increasing $\beta$. We choose the latter option and set $\beta=0.2$, see Figure \ref{fig:2x2_p2}. The results are shown in Table \ref{tab:mult_Schwarz_example_two_stage_strategy_iteration}, where now the iteration error and discretization are balanced and the overall error has decreased.

\begin{table}[H]
\centering
\begin{tabular}{||c|c|c|c|c|c|c|c|c|c||}
\hline
Stage & $N_x$ & $N_y$ & $\beta$ & $K$ & Est. Err. & $\gamma$ &  $e_D^{\{K\}}$ & $\gamma_D$  &    $e_I^{\{K\}}$  \\
\hline
1& 40 & 40 & 0.05 & 2 & 1.23e-03&	1.00e+00&	1.79e-04&	9.99e-01&	1.05e-03	\\
2& 40 & 40 & 0.2  & 2 & 5.04e-04&	1.00e+00&	1.62e-04&	9.99e-01&	3.42e-04	\\
\hline
\end{tabular}
\caption{Two stage solution strategy using multiplicative Schwarz to solve Poisson's equation. Stage 1: $\beta = 0.05$, stage 2: $\beta = 0.2$.}
\label{tab:mult_Schwarz_example_two_stage_strategy_iteration}
\end{table}

\section{Numerical examples for additive Schwarz}
\label{sec:add_Schwarz_numerical_examples}

We repeat analogous numerical examples in \S \ref{sec:mult_Schwarz_numerical_examples} for additive Schwarz. Effectivity ratios for the discretization error and the total error are defined analogously to the case of the multiplicative Schwarz case by replacing $\Phi^{[k+i/p]}$ in the above expressions by $\Phi^{[k]}_i$ in the expressions in \S \ref{sec:mult_Schwarz_error_estimates_effectivity}, where $\Phi^{[k]}_i$ is the numerical approximation to $\phi^{[k]}_i$. A relaxation parameter $\tau=0.4$ was used in all examples. The error estimates are again highly accurate with effectivity ratios close to 1.

\subsection{Estimates for Poisson's equation}
\label{sec:add_Schwarz_examples_poisson}

\subsubsection{$2 \times 1$ subdomains}
\label{sec:add_Schwarz_example_poisson_2x1}

We solve the same problem described in \S \ref{sec:mult_Schwarz_example_poisson_2x1} by equations \eqref{eq:Poissons_equation} and \eqref{eq:QoI_4_Poisson} using additive Schwarz. The results are shown in Table \ref{tab:add_Schwarz_example_poisson_2x1}.  In comparison to the results in \S \ref{sec:mult_Schwarz_example_poisson_2x1}, we observe that the additive Schwarz method has much higher iteration error than multiplicative Schwarz method. The discretization error is of course approximately the same.

\begin{table}[H]
\centering
\begin{tabular}{||c|c|c|c|c|c|c|c|c||}
\hline
$N_x$ & $N_y$ & $\beta$ & $K$ & Est. Err. & $\gamma$ &  $e_D^{\{K\}}$ & $\gamma_D$  &    $e_I^{\{K\}}$  \\
\hline
20 & 20 & 0.1 & 2 & 1.09e-02&	1.00e+00&	4.52e-04&	9.98e-01&	1.05e-02	\\
20 & 20 & 0.2 & 2 & 1.04e-02&	1.00e+00&	4.34e-04&	9.98e-01&	9.96e-03	\\
\hline
20 & 20 & 0.1 & 2 & 1.09e-02&	1.00e+00&	4.52e-04&	9.98e-01&	1.05e-02	\\
20 & 20 & 0.1 & 4 & 4.23e-03&	9.99e-01&	6.02e-04&	9.98e-01&	3.62e-03	\\
\hline
20 & 20 & 0.1 & 2 & 1.09e-02&	1.00e+00&	4.52e-04&	9.98e-01&	1.05e-02	\\
40 & 40 & 0.1 & 2 & 1.06e-02&	1.00e+00&	1.14e-04&	9.99e-01&	1.05e-02	\\
\hline
\end{tabular}
\caption{Additive Schwarz for Poisson's equation: $2 \times 1$ subdomains.}
\label{tab:add_Schwarz_example_poisson_2x1}
\end{table}

\subsubsection{$4 \times 1$ subdomains}
\label{sec:add_Schwarz_example_poisson_4x1}

The results solving the same problem using twice the number of subdomains are shown in Table \ref{tab:add_Schwarz_example_poisson_4x1}. The iteration error is considerably larger than for multiplicative Schwarz and the convergence rate with increasing numbers of iterations appears to be much slower. The discretization error is again approximately the same.

\begin{table}[H]
\centering
\begin{tabular}{||c|c|c|c|c|c|c|c|c||}
\hline
$N_x$ & $N_y$ & $\beta$ & $K$ & Est. Err. & $\gamma$ &  $e_D^{\{K\}}$ & $\gamma_D$  &    $e_I^{\{K\}}$  \\
\hline
20 & 20 & 0.1 & 2 & 1.89e-02&	1.00e+00&	5.42e-04&	9.96e-01&	1.84e-02	\\
20 & 20 & 0.2 &	2 & 1.19e-02&	1.00e+00&	6.04e-04&	9.97e-01&	1.13e-02	\\
\hline
20 & 20 & 0.1 & 2 &	1.89e-02&	1.00e+00&	5.42e-04&	9.96e-01&	1.84e-02	\\
20 & 20 & 0.1 & 4 &	1.21e-02&	1.00e+00&	6.51e-04&	9.97e-01&	1.14e-02	\\
\hline
20 & 20 & 0.1 & 2 &	1.89e-02&	1.00e+00&	5.42e-04&	9.96e-01&	1.84e-02	\\
40 & 40 & 0.1 & 2 &	1.85e-02&	1.00e+00&	1.38e-04&	9.99e-01&	1.84e-02	\\
\hline
\end{tabular}
\caption{Additive Schwarz for Poisson's equation: $4 \times 1$ subdomains.}
\label{tab:add_Schwarz_example_poisson_4x1}
\end{table}

\subsubsection{$4 \times 4$ subdomains}
\label{sec:add_Schwarz_example_poisson_4x4}

Repeating the problem in \S \ref{sec:mult_Schwarz_example_poisson_4x4} and using additive Schwarz produces the results provided in Table \ref{tab:add_Schwarz_example_poisson_4x4}.

\begin{table}[H]
\centering
\begin{tabular}{||c|c|c|c|c|c|c|c|c||}
\hline
$N_x$ & $N_y$ & $\beta$ & $K$ & Est. Err. & $\gamma$ &  $e_D^{\{K\}}$ & $\gamma_D$  &    $e_I^{\{K\}}$  \\
\hline
20 & 20 & 0.1 & 2 &	2.18e-02&	1.00e+00&	6.83e-04&	1.00e+00&	2.12e-02	\\
20 & 20 & 0.2 &	2 & 1.25e-02&	1.00e+00&	6.14e-04&	1.00e+00&	1.18e-02	\\
\hline
20 & 20 & 0.1 & 2 &	2.18e-02&	1.00e+00&	6.83e-04&	1.00e+00&	2.12e-02	\\
20 & 20 & 0.1 & 4 &	1.58e-02&	1.00e+00&	9.42e-04&	9.86e-01&	1.48e-02	\\
\hline
20 & 20 & 0.1 & 2 &	2.18e-02&	1.00e+00&	6.83e-04&	1.00e+00&	2.12e-02	\\
40 & 40 & 0.1 & 2 &	2.13e-02&	1.00e+00&	1.70e-04&	1.00e+00&	2.12e-02	\\
\hline
\end{tabular}
\caption{Additive Schwarz for Poisson's equation: $4 \times 4$ subdomains.}
\label{tab:add_Schwarz_example_poisson_4x4}
\end{table}

Once again the iteration error is significantly greater than in the multiplicative case and appears to improve more slowly with increasing overlap or number of iterations.

\subsection{A convection-diffusion problem}
\label{sec:add_Schwarz_example_convection_diffusion}

The problem formulation is defined in \S \ref{sec:mult_Schwarz_examples_convection_diffusion} by equations \eqref{eq:convection_diffusion} and \eqref{eq:convection_diffusion_QoI}. We provide results for two different configurations of the subdomains in Table \ref{tab:add_Schwarz_example_convection_diffusion} below.

\begin{table}[H]
\centering
\begin{tabular}{||c|c|c|c|c|c||}
\hline
$K$ & Est. Err. & $\gamma$ &  $e_D^{\{K\}}$ & $\gamma_D$  & $e_I^{\{K\}}$  \\
\hline
\multicolumn{6}{||c||}{$4 \times 1$ configuration} \\
\hline
2&	1.78e-02&	1.00e+00&	-1.04e-04&	9.92e-01&	1.79e-02	\\
4&	1.28e-02&	1.00e+00&	-9.38e-05&	9.81e-01&	1.29e-02	\\
6&	8.40e-03&	1.00e+00&	-5.56e-05&	9.54e-01&	8.46e-03	\\
\hline
\multicolumn{6}{||c||}{$1 \times 4$ configuration} \\
\hline
2&	1.08e-02&	1.00e+00&	-1.32e-04&	9.91e-01&	1.10e-02	\\
4&	5.11e-03&	1.00e+00&	-2.37e-04&	9.93e-01&	5.35e-03	\\
6&	2.32e-03&	1.00e+00&	-3.01e-04&	9.94e-01&	2.62e-03	\\
\hline
\end{tabular}
\caption{Additive Schwarz for convection-diffusion:  $N_x = N_y = 20, \ \beta =  0.1$.}
\label{tab:add_Schwarz_example_convection_diffusion}
\end{table}

The differences between these two configurations are not as dramatic as in the case of multiplicative Schwarz. Furthermore, both $4 \times 1$ and $1 \times 4$ configurations had essentially converged after 6 iterations of multiplicative Schwarz. This is far from true for additive Schwarz.

\subsection{Two stage solution strategy for Poisson's equation}
\label{sec:add_Schwarz_example_two_stage_strategy}

\subsubsection{Dominant discretization error}
\label{sec:add_num_exp_disc_err_dom}

We repeat the problem in \S \ref{sec:mult_Schwarz_num_exp_disc_err_dom} and the results are shown in Table \ref{tab:add_Schwarz_example_two_stage_strategy_discretization}. We observe the expected reduction in discretization error in subdomain 4, but the reduction in total error is not as dramatic as in the multiplicative Schwarz case. There is less cancellation between discretization errors of opposite sign  following local mesh refinement, and the iteration error is much larger for additive Schwarz. However, after mesh refinement in subdomain 4 and six iterations, the iteration error makes the largest contribution to the total error.

\begin{table}[H]
\centering
\begin{tabular}{||c|c|c|c|c|c|c||}
\hline
Stage & Num. vertices  & Est. Err. & $\gamma$ &  $e_D^{\{K\}}$ & $\gamma_D$  & $e_I^{\{K\}}$  \\
\hline
1 & 121 & 3.24e-03 & 9.88e-01 & 2.48e-03 & 9.90e-01 & 7.61e-04  \\
2 & 253 & 1.24e-03 & 1.00e+00 & 4.79e-04 & 1.04e+00 & 7.60e-04  \\
\hline
Stage & Num. vertices  & $i$ & 1 & 2 & 3 & 4   \\
\hline
1 & 121 & $S_i^K$ & 6.82e-05 & -3.76e-04 & -3.76e-04 & 3.16e-03 \\
2 & 253 & $S_i^K$ & 4.13e-05 & -1.78e-04 & -1.78e-04 & 7.93e-04 \\
\hline
\end{tabular}
\caption{Two stage solution strategy using additive Schwarz to solve Poisson's equation: $\beta = 0.2, \ K=6$.}
\label{tab:add_Schwarz_example_two_stage_strategy_discretization}
\end{table}

\subsubsection{Dominant iteration error}
\label{sec:add_num_exp_it_err_dom}

The results upon repeating the problem in \S \ref{sec:mult_Schwarz_num_exp_it_err_dom} are shown in Table \ref{tab:add_Schwarz_example_two_stage_strategy_iteration}. Increasing the overlap reduces the iteration error, but not as effectively as for multiplicative Schwarz, and after only two iterations the iteration error remains significantly larger than the discretization error.

\begin{table}[H]
\centering
\begin{tabular}{||c|c|c|c|c|c|c|c|c||}
\hline
$N_x$ & $N_y$ & $\beta$ & $K$ & Est. Err. & $\gamma$ &  $e_D^{\{K\}}$ & $\gamma_D$  &    $e_I^{\{K\}}$  \\
\hline
40 & 40 & 0.05 & 2 & 1.05e-02&	1.00e+00&	1.19e-04&	1.00e+00&	1.04e-02	\\
40 & 40 & 0.2  & 2 & 8.27e-03&	1.00e+00&	1.15e-04&	1.00e+00&	8.15e-03	\\
\hline
\end{tabular}
\caption{Two stage solution strategy using additive Schwarz to solve Poisson's equation. Stage 1: $\beta = 0.05$, stage 2: $\beta = 0.2$.}
\label{tab:add_Schwarz_example_two_stage_strategy_iteration}
\end{table}

%%%%%%%%%%%%%%%%%%%%%%%%%%%%%%%%%%%%%%%%%%%%%%%%%%%%%%%%%%%%%%%%%%%%%%%%%%%%%%%%%%%%%%%%%%%%%%%%%%%%
%                                                                                                  %
%   A posteriori error analysis for multiplicative Schwarz                                         %
%                                                                                                  %
%%%%%%%%%%%%%%%%%%%%%%%%%%%%%%%%%%%%%%%%%%%%%%%%%%%%%%%%%%%%%%%%%%%%%%%%%%%%%%%%%%%%%%%%%%%%%%%%%%%%

\section{Details of analysis: algorithm reformulation, technical lemmas and proofs}
\label{sec:proofs}
% \label{sec:multiplicative_Schwarz_error_analysis}

\subsection{Analogy with algebraic Gauss-Seidel iteration}

It is helpful to consider an algebraic analog of multiplicative Schwarz overlapping domain decomposition in order to provide insight in to the unusual forms of equations \eqref{eq:multiplicative_adj} and \eqref{eq:def_tau_i_Q} and Theorem \ref{thm:multiplicative_discretization_error}. Consider solving the algebraic linear system $Mx=b$, where $M$ is a $p \times p$ matrix, using $K$ Gauss-Seidel iterations. (Here we choose $p=4$ and $K=5$.)

We decompose the matrix $M$ as a sum of strictly lower triangular, diagonal and a strictly upper triangular matrices as $M=L+D+U$ and solve
$$
(L+D)x^{\{k+1\}} = b-Ux^{\{k\}}, k=0, 1, \dots .
$$
Let $A=(L+D)$ and $B=U$, each of which are $p \times p$ matrices. The complete Gauss-Seidel iteration can be written as the following block lower triangular system
\begin{equation*}
C_{\hbox{gs}} \bfx = \bfb,
\end{equation*}
where
\begin{equation*}
C_{\hbox{gs}} = \left[
\begin{matrix}
A &   0  &   0  &   0  &   0   \\
B & A &   0  &   0  &   0   \\
 0   & B & A &   0  &   0   \\
 0   &   0  & B & A &   0   \\
 0   &   0  &   0  & B & A
\end{matrix}
\right ].
\end{equation*}
$C_{\hbox{gs}}$  has a block bandwidth of two. Here,
\begin{equation*}
\bfx^\top = \left( {x^{\{1\}}}^\top, {x^{\{2\}}}^\top, {x^{\{3\}}}^\top, {x^{\{4\}}}^\top, {x^{\{5\}}}^\top \right)
\quad\hbox{and}\quad
\bfb^\top = \left( b^\top, b^\top, b^\top, b^\top, b^\top \right).
\end{equation*}
The corresponding adjoint problem is
$C_{\hbox{gs}}^\top \bfphi = \bfpsi$
where $\bfpsi^\top= \left( 0, 0, 0, 0, \psi^\top \right)$,
and
\begin{equation*}
C_{\hbox{gs}}^\top= \left[
\begin{matrix}
A^\top & B^\top &     0     &   0       &   0       \\
   0      & A^\top & B^\top &   0       &   0       \\
   0      &     0     & A^\top & B^\top &   0       \\
   0      &     0     &     0     & A^\top & B^\top \\
   0      &     0     &     0     &   0       & A^\top
\end{matrix}
\right].
\end{equation*}
Note that the adjoint is non-zero only for the final solutions $x^{\{K\}}=x^{\{5\}}$. Let
\begin{equation*}
\bfphi^\top= \left( \bfphi_1, \bfphi_2, \bfphi_3, \bfphi_4, \bfphi_5 \right)^\top.
\end{equation*}
The adjoint problems are
\begin{equation*}
\begin{aligned}
&A^\top \bfphi_5 = \psi, \; \\
&A^\top \bfphi_4 = -B^\top \bfphi_5,\;
A^\top \bfphi_3 = -B^\top \bfphi_4,\;
A^\top \bfphi_2 = -B^\top \bfphi_3,\;
A^\top \bfphi_1 = -B^\top \bfphi_2,
\end{aligned}
\end{equation*}
which can be solved sequentially by backward substitution. Recall that $A$ is lower triangular and $B$ is strictly upper triangular, hence
\begin{equation*}
C_{\hbox{gs}} = \left[
\begin{array}{cccc:cccc:c}
A_{11} &        &        &        &        &        &        &        &        \\
A_{21} & A_{22} &        &        &        &        &        &        &        \\
A_{31} & A_{32} & A_{33} &        &        &        &        &        &        \\
A_{41} & A_{42} & A_{43} & A_{44} &        &        &        &        &        \\
\hdashline
       & A_{12} & A_{13} & A_{14} & A_{11} &        &        &        &        \\
       &        & A_{23} & A_{24} & A_{21} & A_{22} &        &        &        \\
       &        &        & A_{34} & A_{31} & A_{32} & A_{33} &        &        \\
       &        &        &        & A_{41} & A_{42} & A_{43} & A_{44} &        \\
\hdashline
       &        &        &        &        &        &        & \vdots & \dots
\end{array} \right],
\end{equation*}
\begin{equation*}
C_{\hbox{gs}}^\top = \left[
\begin{array}{cccc:cccc:c}
A_{11}^\top & A_{21}^\top & A_{31}^\top & A_{41}^\top &             &             &             &             &        \\
            & A_{22}^\top & A_{32}^\top & A_{42}^\top & A_{12}^\top &             &             &             &        \\
            &             & A_{33}^\top & A_{43}^\top & A_{13}^\top & A_{23}^\top &             &             &        \\
            &             &             & A_{44}^\top & A_{14}^\top & A_{24}^\top & A_{34}^\top &             &        \\
\hdashline
            &             &             &             & A_{11}^\top & A_{21}^\top & A_{31}^\top & A_{41}^\top &        \\
            &             &             &             &             & A_{22}^\top & A_{32}^\top & A_{42}^\top &        \\
            &             &             &             &             &             & A_{33}^\top & A_{43}^\top &        \\
            &             &             &             &             &             &             & A_{44}^\top &        \\
 \hdashline
           &             &             &             &             &             &             & \vdots      & \dots
\end{array} \right],
\end{equation*}
and the adjoint equations within each block can also be solved via backward substitution. The adjoint equations are therefore
\begin{eqnarray}
&&A_{ii}^\top \phi_i^{[K]} = \psi_i -\sum_{j>i} A^\top_{ji}\phi_j^{[K]},
  \quad i=1, \dots p, \label{eq:algebraic_GS_adjoint_a} \\
&&A_{ii}^\top \phi_i^{[k]} = -\sum_{j<i} A^\top_{ji}\phi_j^{[k+1]} - \sum_{j>i} A^\top_{ji}\phi_j^{[k]},
\quad  i=1, \dots p, \quad k = K-1, \dots, 1. \label{eq:algebraic_GS_adjoint_b}
\end{eqnarray}
The form of the adjoint problems in \eqref{eq:algebraic_GS_adjoint_a} and \eqref{eq:algebraic_GS_adjoint_b} mimic those in equations \eqref{eq:multiplicative_adj} and \eqref{eq:def_tau_i_Q}. The sum on the RHS of \eqref{eq:algebraic_GS_adjoint_a} represents the additional adjoint problems that must be solved to estimate effect of errors made while solving forward problems during the final ($Kth$) iteration. We call these \emph{within iteration} transfer errors. The first sum on the RHS of \eqref{eq:algebraic_GS_adjoint_b} represents the additional adjoint problems that must be solved to estimate the effect of errors made while solving forward problems during the previous iteration. We call these \emph{between iteration} transfer errors. The second sum on the RHS of \eqref{eq:algebraic_GS_adjoint_b} again represents additional adjoint problems to estimate \emph{within iteration} transfer errors. These two distinct types of transfer error were earlier identified in the context of operator decomposition approaches to coupled semilinear elliptic systems in \cite{carey2013posteriori}.  Since $C_{\hbox{gs}} \bfx = \bfb$ is just a $Kp \times Kp$ linear system,
\begin{equation}
\label{eq:algebraic_GS_error_representation}
(\bfe,\bfpsi) = (\calR,\bfphi)
\end{equation}
where
\begin{equation*}
  \calR = \bfb - C_{\hbox{gs}} \hat{\bfx}
\end{equation*}
for approximate solution $\hat{\bfx}$. The error in the quantity of interest can be expressed as an inner products of two vectors of length $Kp$, and \eqref{eq:algebraic_GS_error_representation} mimics the result in equation \eqref{eq:multiplicative_disc_err_rep}.

\subsection{Details of analysis of multiplicative Schwarz algorithm}
\label{sec:multiplicative_Schwarz_error_analysis}

\subsubsection{Reformulation of the algorithm}
Algorithm \ref{alg:multiplicative_basic} is not amenable to adjoint based analysis since the affine solution space $H^{1}_{D_{k}}(\Omega_i)$ changes at every iteration. We reformulate the algorithm  by using a standard lifting technique to account for this in Algorithm~\ref{alg:multiplicative_reformulated}. We set
\begin{equation}
\label{eq:def_w_mul_sch}
\widetilde{u}^{\{k+i/p\}} = w^{\{k+i/p\}}  + u^{\{k+(i-1)/p\}} \quad \text{ on } \Omega_i,
\end{equation}
where $w^{\{k+i/p\}} \in H_0^{1}(\Omega_i)$.

\begin{algorithm}[H]
\caption{Reformulated overlapping multiplicative Schwarz}
\label{alg:multiplicative_reformulated}
\begin{algorithmic}
\State Given $u^{\{0\}}$ defined on $\Omega$
\For{$k=0, 1, 2, \dots, K-1$ }
  \For{$i= 1, 2, \dots, p$ }
    % \State \\
    \State Find $w^{\{k+i/p\}} \in H_0^{1}(\Omega_1)$ such that \\
    \begin{equation}
    \label{eq:multiplicative_wkph}
        a_i \left( w^{\{k+i/p\}},v \right)= l_i(v)- a_i \left(u^{\{k + (i-1)/p\}},v \right), \quad \forall v \in H_0^{1}(\Omega_i).	
    \end{equation}
    % \State \\
    \State Let
    \begin{equation}
    \label{eq:multiplicative_ukph}
      u^{\{k+i/p\}} = \left \{ \begin{gathered} \begin{aligned}
      &u^{\{k+ (i-1)/p\}}+w^{\{k+i/p\}}, \; &&\hbox{on } \overline{\Omega}_i, \\
      &u^{\{k+ (i-1)/p\}},                  &&\hbox{on } \Omega \backslash \overline{\Omega}_i.
      \end{aligned} \end{gathered}  \right .
    \end{equation}
    % \State \\
  \EndFor
\EndFor
\end{algorithmic}
\end{algorithm}

There is an equivalent reformulation of the discrete Algorithm~\ref{alg:multiplicative_reformulated} and we denote the unknown solutions  as $W^{\{k+i/p\}}$  belonging to the spaces $V_{i,h,0} \subset H_0^1(\Omega_i)$. The solutions $W^{\{k+i/p\}}$  are defined formally for the analysis but are not computed in practice.

To distinguish between different solutions (true, analytical, discrete) we use the notation in Table \ref{tab:multiplicative_Schwarz_spaces}.

\begin{table}[H]
\centering
\begin{tabular}{||c|c|c|l||}
\hline
Notation &  Formula       &Space & Meaning \\
\hline
$u$                        && $H^1_0(\Omega)$              & True solution \\
$u^{\{k\}}$                && $H^{1}_{0}(\Omega_i)$        & Global analytic solution at iteration $k$ \\
${U}^{\{k\}}$              && $V_{h}$                      & Global discrete solution at iteration $k$ \\
\hline
$\widetilde{u}^{\{k + i/p\}}$ && $H^{1}_{D_{k}}(\Omega_i)$  & Analytic solutions on $\Omega_i$ at iteration $k$ \\
$\widetilde{U}^{\{k+ i/p\}}$  && $V_{i,h}^{k}$              & Discrete solutions on $\Omega_i$ at iteration $k$ \\
\hline
$w^{\{k + i/p\}}$            && $H^{1}_{0}(\Omega_i)$      & Analytic solns on $\Omega_i$ with homogen. bcs at iteration $k$ \\
$W^{\{k + i/p\}}$            && $V_{i,h,0}$                & Discrete solns on $\Omega_i$ with homogen. bcs at iteration $k$ \\
\hline
$e^{\{k\}}$         &$u-U^{\{k\}}$                    & $H^{1}_{0}(\Omega)$      & Total error \\
%$\bar{e}^{\{k\}} = u - {u}^{\{k\}}$                        & $H^{1}_{0}(\Omega)$      & Iteration error\\
%$\hat{e}^{\{k\}} = {u}^{\{k\}} - {U}^{\{k\}}$              & $H^{1}_{0}(\Omega)$      & Discretization error\\
%$\widetilde{e}_i^{\{k\}} = {w}_i^{\{k\}} - {W}_i^{\{k\}}$  & $H^{1}_{0}(\Omega_i)$    & Local domain error\\
$e_I^{\{k\}}$       &$u-u^{\{k\}}$                    & $H^{1}_{0}(\Omega)$      & Global iteration error at iteration $k$  \\
$e_D^{\{k\}}$       &$u^{\{k\}}-U^{\{k\}}$            & $H^{1}_{0}(\Omega)$      & Global discretization error at iteration $k$ \\
\hline
$e_{W}^{\{k + i/p\} }$  &$w^{\{k+i/p\}}-W^{\{k+i/p\}}$ & $H^{1}_{0}(\Omega_i)$   & Discretization error on $\Omega_i$ with homogen. bcs at it. $k$ \\
\hline
\end{tabular}
\caption{Multiplicative Schwarz: notation for different solutions and their spaces. }
\label{tab:multiplicative_Schwarz_spaces}
\end{table}

\subsubsection{Technical lemmas}
\label{sec:multiplicative_Schwarz_lemmas}

Let $e_W^{\{k\}} = w^{\{k\}} - W^{\{k\}}$. By \eqref{eq:def_w_mul_sch} we have
% for $k \geq 1$,
\begin{equation}
\label{eq:multiplicative_err_tilde_hat_rel}
\begin{aligned}
e_W^{\{k+i/p\}} &= e_D^{\{k+i/p\}} - e_D^{\{k+(i-1)/p\}} \quad \text{on } \Omega_i.
\end{aligned}	
\end{equation}
Note that $e_W^{\{k+i/p\}} = 0$ on $\partial \Omega_i$. We set $e_W^{\{k+i/p\}} = 0$ on $\Omega \setminus \Omega_i$.

%%%%%%%%%%%%%%%%%%%%%%%%%%%%%%%%%%%%%%%%%%%%%%%%%%%%%%%%%%%%%%%%%%%%%%%%%%%%%%%%%%%%%%%%%%%%%%%%%%%%
%                                                                                                  %
%  Lemma One                                                                                       %
%                                                                                                  %
%%%%%%%%%%%%%%%%%%%%%%%%%%%%%%%%%%%%%%%%%%%%%%%%%%%%%%%%%%%%%%%%%%%%%%%%%%%%%%%%%%%%%%%%%%%%%%%%%%%%

\begin{lemma}[Error in QoI in terms of discretization errors with homogeneous bcs]
\label{lem:multiplicative_disc_err_decomp}

The discretization error in the QoI is
\begin{equation}
\label{eq:mul_sch_err_disc_exp}
\left(e_D^{\{K\}},\psi\right) =  \sum_{k=0}^{K-1} \sum_{i=1}^p \sum_{j=1}^p \left ( e_W^{\{k+i/p\}}, \psi_j \right )_{ij}.
\end{equation}
\end{lemma}

\begin{proof}
From equation \eqref{eq:multiplicative_err_tilde_hat_rel} and the fact that $\psi_j = 0$ on $\Omega \setminus \Omega_j$ for fixed $j$ we have
\begin{equation*}
\begin{aligned}
\left(e_D^{\{K\}},\psi_j\right) &= \left(e_D^{\{K-1+p/p\}},\psi_j\right) \\
&= \left(e_W^{\{K-i+p/p\}},\psi_j\right)_{pj} + \left(e_D^{\{K-1+(p-1)/p\}},\psi_j\right) \\
&= \left(e_W^{\{K-i+p/p\}},\psi_j\right)_{pj} + \left(e_W^{\{K-i+(p-1)/p\}},\psi_j\right)_{(p-1)j} +  \left(e_D^{\{K-1+(p-2)/p\}},\psi_j\right). \\
\end{aligned}
\end{equation*}
Continuing,
\begin{equation*}
\left(e_D^{\{K\}},\psi_j\right) =  \left(e_D^{\{K-1\}},\psi_j\right) + \sum_{i=1}^p \left(e_W^{\{K - 1 + i/p\}},\psi_j\right)_{ij}.
\end{equation*}
This is a recursive relation for $e_D^{\{K\}}$. Expanding $\left(e_D^{\{K-1\}},\psi_j\right)$ as above leads to
\begin{equation*}
\begin{aligned}
\left(e_D^{\{K\}},\psi_j\right) &= \sum_{k=0}^{K-1} \sum_{i=1}^p \ \left(e_W^{\{k+i/p\}},\psi_j\right)_{ij}.
\end{aligned}
\end{equation*}
Summing over $j=1, \dots, p$,
\begin{equation*}
\sum_{j=1}^p \left(e_D^{\{K\}},\psi_j\right) = \sum_{k=0}^{K-1} \sum_{i=1}^p \sum_{j=1}^p \ \left(e_W^{\{k+i/p\}},\psi_j\right)_{ij}.
\end{equation*}
\end{proof}

%%%%%%%%%%%%%%%%%%%%%%%%%%%%%%%%%%%%%%%%%%%%%%%%%%%%%%%%%%%%%%%%%%%%%%%%%%%%%%%%%%%%%%%%%%%%%%%%%%%%
%                                                                                                  %
%  Lemma Two                                                                                       %
%                                                                                                  %
%%%%%%%%%%%%%%%%%%%%%%%%%%%%%%%%%%%%%%%%%%%%%%%%%%%%%%%%%%%%%%%%%%%%%%%%%%%%%%%%%%%%%%%%%%%%%%%%%%%%

\begin{lemma}[Bilinear form with discretization errors with homogeneous bcs]
\label{lem:multiplicative_etlil_expansion}
For any $v \in H^{1}_{D_{k}}(\Omega_i)$ we have
\begin{equation}
\label{eq:multiplicative_etil_lev_0}
  a_i\left(e_W^{\{i/p\}},v\right) = a_i\left(e_D^{\{i/p\}},v\right) - \sum_{r = 1}^{i-1}a_{ir}\left(e_W^{\{r/p\}},v\right).
\end{equation}
% \begin{equation}
% \label{eq:etil_lev_1}
% a_i(e_W^{\{1 + i/p\}},v) = a_i(e_D^{\{1 + i/p\}},v) - \sum_{r = 1}^{i-1}a_{ir}(e_W^{\{1 + r/p\}},v),
% - \sum_{r = i+1}^{p}a_{ir}(e_W^{\{ r/p\}},v) -a_i(e_D^{\{ i/p\}},v)
% \end{equation}
and for $k \geq 1$,
\begin{equation}
\label{eq:multiplicative_etil_lev_k}
  a_i\left(e_W^{\{k+i/p\}},v\right) = a_i\left(e_D^{\{k + i/p\}},v\right) - a_i\left(e_D^{\{k -1 + i/p\}},v\right)
  -\sum_{r = 1}^{i-1} a_{ir}\left(e_W^{\{k + r/p\}},v\right) - \sum_{r=i+1}^{p}a_{ir}\left(e_W^{\{k -1 + r/p\}} ,v\right).
\end{equation}
\end{lemma}

\begin{proof}
By \eqref{eq:multiplicative_err_tilde_hat_rel} we have for $m < i$,
\begin{equation*}
\begin{aligned}
a_i\left(e_D^{\{m/p\}},v\right)
&=  a_i\left(e_D^{\{(m-1)/p\}},v\right) + a_{i,m}\left(e_W^{\{m/p\}},v\right),   \\
&=  a_i\left(e_D^{\{(m-2)/p\}},v\right) + a_{i,m-1}\left(e_W^{\{(m-1)/p\}},v\right) + a_{i,m}\left(e_W^{\{m/p\}},v\right).
\end{aligned}
\end{equation*}
where we use $e_W^{\{r/p\}} = 0$ on $\Omega \setminus \Omega_r$.
Continuing in this manner yields
\begin{equation}
\label{eq:partial_result_sum_hat_e_mul_sch}
a_i\left(e_D^{\{m/p\}},v\right)
= a_i\left(e_D^{\{0\}},v\right) + \sum_{r=1}^m a_{ir}\left(e_W^{\{r/p\}},v\right)
= \sum_{r=1}^m a_{ir}\left(e_W^{\{r/p\}},v\right),
\end{equation}
since $e_D^{\{0\}} = 0$. Again by \eqref{eq:multiplicative_err_tilde_hat_rel},
\begin{equation}
\label{eq:mul_sch_wide_e_one_expansion}
a_i\left(e_W^{\{i/p\}},v\right) = a_i\left(e_D^{\{i/p\}},v\right) - a_i\left(e_D^{\{(i-1)/p\}},v\right).
\end{equation}
Using \eqref{eq:partial_result_sum_hat_e_mul_sch} with $m=i-1$ with \eqref{eq:mul_sch_wide_e_one_expansion} leads to
% \begin{equation}
% \begin{aligned}
% &= a_i(e_D^{\{i/p\}},v) - a_{i,i-1}(e_W^{\{(i-1)/p\}},v)  - a_i(e_D^{\{(i-2)/p\}},v)\\
% \end{aligned}
% \end{equation}
\begin{equation*}
\begin{aligned}
a_i\left(e_W^{\{i/p\}},v\right) &= a_i\left(e_D^{\{i/p\}},v\right) - \sum_{r = 1}^{i-1}a_{ir}\left(e_W^{\{r/p\}},v\right),
\end{aligned}
\end{equation*}
thus showing \eqref{eq:multiplicative_etil_lev_0}. A similar argument shows \eqref{eq:multiplicative_etil_lev_k} for $k \geq 1$.

%  Now let $k \geq 2$. Following as above we have,
% \begin{equation}
% \label{eq:multiplicative_etil_lev_k_a}
% 	a_i(e_W^{\{k+i/p\}},v)  = a_i(e_D^{\{k+i/p\}},v) - \sum_{j = 1}^{i-1}a_{ij}(e_W^{\{k + j/p\}},v)
% - \sum_{j = i+1}^{p}a_{ij}(e_W^{\{ k - 1 + j/p\}},v) -a_i(e_D^{\{ k - 1 + i/p\}},v).
% \end{equation}

\end{proof}

%\begin{corr}
%Rearranging the final expressions to highlight the accumulation of error during a single ``sweep'',
%$$
%a_i(e_D^{\{i/p\}},v) = a_i(e_W^{\{i/p\}},v) + \sum_{j = 1}^{i-1}a_{ij}(e_W^{\{j/p\}},v),
%$$
%and
%$$	a_i(e_D^{\{k+i/p\}},v)
%= a_i(e_D^{\{ k - 1 + i/p\}},v)
% + a_i(e_W^{\{k+i/p\}},v)
% + \sum_{j = i+1}^{p}a_{ij}(e_W^{\{ k - 1 + j/p\}},v)
% + \sum_{j = 1}^{i-1}a_{ij}(e_W^{\{k + j/p\}},v).
%$$
%\end{corr}

%%%%%%%%%%%%%%%%%%%%%%%%%%%%%%%%%%%%%%%%%%%%%%%%%%%%%%%%%%%%%%%%%%%%%%%%%%%%%%%%%%%%%%%%%%%%%%%%%%%%
%                                                                                                  %
%  Lemma Three                                                                                      %
%                                                                                                  %
%%%%%%%%%%%%%%%%%%%%%%%%%%%%%%%%%%%%%%%%%%%%%%%%%%%%%%%%%%%%%%%%%%%%%%%%%%%%%%%%%%%%%%%%%%%%%%%%%%%%

\begin{lemma}[Sums of bilinear form with discretization errors with homogeneous bcs]
\label{lem:multiplicative_etil_sum}
For $0 \leq Q \leq K-1$ we have
\begin{equation} \label{eq:multiplicative_etil_sum}
\sum_{k=0}^{Q}a_i\left(e_W^{\{k+i/p\}},v\right)
= a_i\left(e_D^{\{Q+i/p\}},v\right)
  - \sum_{k=0}^{Q} \sum_{r=1}^{i-1} a_{ir}\left(e_W^{\{k+r/p\}},v\right)
  - \sum_{k=0}^{Q-1} \sum_{r=i+1}^{p} a_{ir}\left(e_W^{\{k+r/p\}},v\right).
\end{equation}
\end{lemma}

\begin{proof}
By Lemma~\ref{lem:multiplicative_etlil_expansion},
\begin{equation*}
\begin{aligned}
\sum_{k=0}^{Q}a_i & \left(e_W^{\{k+i/p\}},v\right)
= \sum_{k=1}^{Q}a_i\left(e_W^{\{k+i/p\}},v\right) + a_i\left(e_W^{\{i/p\}},v\right)\\
&= \sum_{k=1}^{Q} \left\{a_i\left(e_D^{\{k + i/p\}},v\right) - a_i\left(e_D^{\{k -1 + i/p\}},v\right)
  -\sum_{r = 1}^{i-1} a_{ir}\left(e_W^{\{k + r/p\}},v\right)
  -\sum_{r=i+1}^{p} a_{ir}\left(e_W^{\{k -1 + r/p\}} ,v\right)\right\}  \\
&\qquad + a_i\left(e_D^{\{i/p\}},v\right) - \sum_{r = 1}^{i-1}a_{ir}\left(e_W^{\{r/p\}},v\right) \\
&= \sum_{k=1}^{Q} \left\{a_i\left(e_D^{\{k + i/p\}},v\right) - a_i\left(e_D^{\{k -1 + i/p\}},v\right) \right\} +  a_i\left(e_D^{\{i/p\}},v\right)\\
&\qquad -\sum_{k=1}^{Q} \sum_{r = 1}^{i-1} a_{ir}\left(e_W^{\{k + r/p\}},v\right)
        -\sum_{r = 1}^{i-1}a_{ir}\left(e_W^{\{r/p\}},v\right)
        -\sum_{k=1}^{Q} \sum_{r=i+1}^{p} a_{ir}\left(e_W^{\{k -1 + r/p\}} ,v\right) \\
% &= \sum_{k=1}^{Q} \left[a_i(e_D^{\{k + i/p\}},v) - a_i(e_D^{\{k -1 + i/p\}},v) \right] +  a_i(e_D^{\{i/p\}},v)\\
% &-\sum_{k=1}^{Q} \sum_{r = 1}^{i-1} a_{ir}(e_W^{\{k + r/p\}},v) - \sum_{r = 1}^{i-1}a_{ir}(e_W^{\{r/p\}},v)
% - \sum_{k=1}^{Q} \sum_{r=i+1}^{p} a_{ir}(e_W^{\{k -1 + r/p\}} ,v) \\
&=a_i\left(e_D^{\{Q+i/p\}},v\right)
  - \sum_{k=0}^{Q} \sum_{r=1}^{i-1} a_{ir}\left(e_W^{\{k+r/p\}},v\right)
  - \sum_{k=0}^{Q-1} \sum_{r=i+1}^{p} a_{ir}\left(e_W^{\{k+r/p\}},v\right).
\end{aligned}
\end{equation*}
%Expanding the sum proves the lemma.
\end{proof}

%\begin{corr}
%Rearranging the final expression to highlight the accumulation of error,
%$$
%a_i(e_D^{\{M+i/p\}},v) =
%  \sum_{k=0}^{M}a_i(e_W^{\{k+i/p\}},v)
%+ \sum_{k=0}^{M} \sum_{j=1}^{i-1} a_{ij}(e_W^{\{k+j/p\}},v)
%+ \sum_{k=0}^{M-1} \sum_{j=i+1}^{p} a_{ij}(e_W^{\{k+j/p\}},v).
%$$
%\end{corr}

%%%%%%%%%%%%%%%%%%%%%%%%%%%%%%%%%%%%%%%%%%%%%%%%%%%%%%%%%%%%%%%%%%%%%%%%%%%%%%%%%%%%%%%%%%%%%%%%%%%%
%                                                                                                  %
%    A posteriori analysis for discretization error                                                %
%                                                                                                  %
%%%%%%%%%%%%%%%%%%%%%%%%%%%%%%%%%%%%%%%%%%%%%%%%%%%%%%%%%%%%%%%%%%%%%%%%%%%%%%%%%%%%%%%%%%%%%%%%%%%%

%%%%%%%%%%%%%%%%%%%%%%%%%%%%%%%%%%%%%%%%%%%%%%%%%%%%%%%%%%%%%%%%%%%%%%%%%%%%%%%%%%%%%%%%%%%%%%%%%%%%
%                                                                                                  %
%  Lemma Four                                                                                     %
%                                                                                                  %
%%%%%%%%%%%%%%%%%%%%%%%%%%%%%%%%%%%%%%%%%%%%%%%%%%%%%%%%%%%%%%%%%%%%%%%%%%%%%%%%%%%%%%%%%%%%%%%%%%%%

\begin{lemma}[Sum of RHS of the adjoint equations over iterations]
\label{lem:mul_sch_ana_lem_5}
Let $2 \leq M \leq p+1$ and $R = M-1$ and $0 \leq Q < K$. Then
\begin{equation*}
\begin{aligned}
&\sum_{k=0}^{Q} \tau_{R}^{Q}\left(e_W^{\{k + R/p\}}\right)
     - \sum_{k=0}^{Q} \sum_{j=M}^{p}  a_{Rj}\left(e_W^{\{k + R/p\}}, \phi^{[Q+j/p]}\right) \\
&=a_R\left(e_D^{\{Q+R/p\}}, \phi^{[Q+R/p]}\right)
     - \sum_{k=0}^{Q} \sum_{i=1}^{R-1} a_{iR}\left(e_W^{\{k+i/p\}}, \phi^{[Q+R/p]}\right)
     - \sum_{k=0}^{Q-1} \sum_{i=M}^{p} a_{iR}\left(e_W^{\{k+i/p\}}, \phi^{[Q+R/p]}\right).
\end{aligned}
\end{equation*}
\end{lemma}

\begin{proof}
From the adjoint equation \eqref{eq:multiplicative_adj} we have
\begin{equation}
\label{eq:mul_lem_5_eq_1}
a_{R}\left(e_W^{\{k + R/p\}}, \phi^{[Q+R/p]}\right)	
=\tau_{i}^{Q}\left(e_W^{\{k + R/p\}}\right) - \sum_{j=M}^{p}  a_{Ri}\left(e_W^{\{k + R/p\}}, \phi^{[Q+j/p]}\right).
\end{equation}
From Lemma~\ref{lem:multiplicative_etil_sum},
\begin{equation}
\label{eq:mul_lem_5_eq_2}
\begin{aligned}
&\sum_{k=0}^{Q}	a_{R}\left(e_W^{\{k + R/p\}}, \phi^{[Q+R/p]}\right)	 \\
&= a_R\left(e_D^{\{Q+R/p\}}, \phi^{[Q+R/p]}\right)
  - \sum_{k=0}^{Q} \sum_{i=1}^{R-1} a_{iR}\left(e_W^{\{k+i/p\}}, \phi^{[Q+R/p]}\right)
  - \sum_{k=0}^{Q-1} \sum_{i=M}^{p} a_{iR}\left(e_W^{\{k+i/p\}}, \phi^{[Q+R/p]}\right).
\end{aligned}
\end{equation}
Combining \eqref{eq:mul_lem_5_eq_1} and \eqref{eq:mul_lem_5_eq_2} proves the result.

\end{proof}

%%%%%%%%%%%%%%%%%%%%%%%%%%%%%%%%%%%%%%%%%%%%%%%%%%%%%%%%%%%%%%%%%%%%%%%%%%%%%%%%%%%%%%%%%%%%%%%%%%%%
%                                                                                                  %
%  Lemma Five                                                                                     %
%                                                                                                  %
%%%%%%%%%%%%%%%%%%%%%%%%%%%%%%%%%%%%%%%%%%%%%%%%%%%%%%%%%%%%%%%%%%%%%%%%%%%%%%%%%%%%%%%%%%%%%%%%%%%%

\begin{lemma}[Sum of RHS of the adjoint equations over iterations and subdomains]
\label{lem:ms_eK_a}
Let $ 1 \leq M \leq p+1$ and $0 \leq Q < K$. Then,
\begin{equation}
\label{eq:mul_sch_lem_6_eq_1}
\begin{aligned}
I= \sum_{k=0}^{Q} \sum_{i=1}^{p} \tau_{i}^{Q}\left(e_W^{\{k + i/p\}}\right)
&= \sum_{i=M}^p  \bilinhatenopar{i}{Q}{Q} + \sum_{k=0}^{Q} \sum_{i=1}^{M-1} \tau_{i}^{Q}\left(e_W^{\{k + i/p\}}\right) \\
&- \sum_{k=0}^{Q}  \sum_{i=1}^{M-1} \sum_{j=M}^{p}\bilintilbothnopar{i}{j}{k}{Q} \\
&- \sum_{k=0}^{Q-1} \sum_{i=M+1}^{p}  \sum_{j=M}^{i-1}  \bilintilbothnopar{i}{j}{k}{Q}.
\end{aligned}
\end{equation}
\end{lemma}

\begin{proof}
The proof is by induction on $M$.
%Let
%\begin{equation}
%	I = \sum_{k=0}^{Q} \sum_{i=1}^{p} \tau_{i}^{Q} (e_W^{\{k + i/p\}}).
%\end{equation}

\medskip
\noindent {\tt [I]} For $M = p+1$ the right-hand side of \eqref{eq:mul_sch_lem_6_eq_1} is simply $I$.

\medskip
\noindent {\tt [II]} Assume that the expression holds for some $2 \leq M \leq p$.

\medskip
\noindent {\tt [III]} To show the result is true for $M=p-1$, we isolate terms involving $e_W^{\{k + (M-1)/p\}}$.

\begin{equation}
\label{eq:mul_sch_lem_6_eq_2}
\begin{aligned}
I
&= \sum_{i=M}^p  \bilinhatenopar{i}{Q}{Q}
 + \sum_{k=0}^{Q} \sum_{i=1}^{M-2} \tau_{i}^{Q} \left(e_W^{\{k + i/p\}}\right)
 - \sum_{k=0}^{Q} \sum_{i=1}^{M-2} \sum_{j=M}^{p}\bilintilbothnopar{i}{j}{k}{Q} \\
&- \sum_{k=0}^{Q-1} \sum_{i=M+1}^{p} \sum_{j=M}^{i-1}  \bilintilbothnopar{i}{j}{k}{Q}
 + \sum_{k=0}^{Q} \tau_{M-1}^{Q}\left(e_W^{\{k + (M-1)/p\}}\right) \\
& - \sum_{k=0}^{Q}  \sum_{j=M}^{p}\bilintilsecondnopar{M-1}{j}{k}{Q}.
\end{aligned}
\end{equation}
From Lemma~\ref{lem:mul_sch_ana_lem_5},
\begin{equation}
\label{eq:mul_sch_lem_6_eq_3}
\begin{aligned}
&\sum_{k=0}^{Q}  \tau_{M-1}^{Q}\left(e_W^{\{k + (M-1)/p\}}\right) - \sum_{k=0}^{Q}  \sum_{j=M}^{p}\bilintilsecondnopar{M-1}{j}{k}{Q} \\
&= \bilinhate{M-1}{Q}{Q}  - \sum_{k=0}^{Q} \sum_{i = 1}^{M-2} \bilintilfirstnopar{i}{M-1}{k}{Q}\\
&- \sum_{k=0}^{Q-1} \sum_{i=M}^{p} \bilintilfirstnopar{i}{M-1}{k}{Q}.
\end{aligned}
\end{equation}
Combining \eqref{eq:mul_sch_lem_6_eq_3} with \eqref{eq:mul_sch_lem_6_eq_2},
\begin{equation*}
\begin{aligned}
I
&= \sum_{i=M-1}^p  \bilinhatenopar{i}{Q}{Q}
 + \sum_{k=0}^{Q} \sum_{i=1}^{M-2} \tau_{i}^{Q} \left(e_W^{\{k + i/p\}}\right)
 - \sum_{k=0}^{Q}  \sum_{i=1}^{M-2} \sum_{j=M-1}^{p}\bilintilbothnopar{i}{j}{k}{Q}\\
&- \sum_{k=0}^{Q-1} \sum_{i=M}^{p}  \sum_{j=M-1}^{i-1}  \bilintilbothnopar{i}{j}{k}{Q}.
\end{aligned}
\end{equation*}
%The only "hard" part was the last term
\end{proof}

\begin{corollary}
\label{corr:1}
Let $0 \leq Q < K$. Then we have
\begin{equation}
\label{eq:mul_sch_corr_1_eq_1}
\begin{aligned}
\sum_{k=0}^{Q} \sum_{i=1}^{p} \tau_{i}^{Q}\left(e_W^{\{k + i/p\}}\right)
&= \sum_{i=1}^p  \bilinhatenopar{i}{Q}{Q}
 + \sum_{k=0}^{Q-1} \sum_{i=1}^{p} \tau_{i}^{Q-1}\left(e_W^{\{k + i/p\}}\right).
\end{aligned}
\end{equation}
\end{corollary}

\begin{proof}
Set $M = 1$ in Lemma~\ref{lem:ms_eK_a} to get,
\begin{equation}
\label{eq:mul_sch_corr_1_eq_2}
\begin{aligned}
\sum_{k=0}^{Q} \sum_{i=1}^{p} \tau_{i}^{Q}\left(e_W^{\{k + i/p\}}\right)
&= \sum_{i=1}^p  \bilinhatenopar{i}{Q}{Q}
 - \sum_{k=0}^{Q-1} \sum_{i=2}^{p}  \sum_{j=1}^{i-1}  \bilintilbothnopar{i}{j}{k}{Q}\\
&= \sum_{i=1}^p  \bilinhatenopar{i}{Q}{Q}
 - \sum_{k=0}^{Q-1} \sum_{i=2}^{p} \tau_{i}^{Q-1}\left(e_W^{\{k + i/p\}}\right)\\
&= \sum_{i=1}^p  \bilinhatenopar{i}{Q}{Q} - \sum_{k=0}^{Q-1} \sum_{i=1}^{p} \tau_{i}^{Q-1}\left(e_W^{\{k + i/p\}}\right),
\end{aligned}
\end{equation}
where we use \eqref{eq:def_tau_i_Q} and note that $\tau_{1}^{Q} (v) = 0$ for $Q < K-1$.
\end{proof}

%%%%%%%%%%%%%%%%%%%%%%%%%%%%%%%%%%%%%%%%%%%%%%%%%%%%%%%%%%%%%%%%%%%%%%%%%%%%%%%%%%%%%%%%%%%%%%%%%%%%
%                                                                                                  %
%  Theorem: Discretization error representation for multiplicative Schwarz                         %
%                                                                                                  %
%%%%%%%%%%%%%%%%%%%%%%%%%%%%%%%%%%%%%%%%%%%%%%%%%%%%%%%%%%%%%%%%%%%%%%%%%%%%%%%%%%%%%%%%%%%%%%%%%%%%

% \begin{theorem}[Discretization error for multiplicative Schwarz]
% \label{thm:multiplicative_discretization_error}
% \begin{equation}
% \label{eq:multiplicative_disc_err_rep}
% (\psi, u^{\{K\}} - U^{\{K\}}) = \sum_{k=0}^{K-1} \sum_{i=1}^p  R_i(\widetilde{U}^{\{k+i/p\}}, \phi^{[k+i/p]}).
% \end{equation}
% \end{theorem}

\subsubsection{Proof of Theorem~\ref{sec:multiplicative_Schwarz_theorem}}
\label{sec:proof_mult_schwarz_disc_err}

\begin{proof}
From Lemma~\ref{lem:multiplicative_disc_err_decomp} and \eqref{eq:def_tau_i_Q},
\begin{equation*}
\left(e_D^{\{K\}},\psi \right) =  \sum_{k=0}^{K-1} \sum_{i=1}^p \sum_{j=1}^p \ \left( e_W^{\{k+i/p\}}, \psi_j\right)_{ij}
= \sum_{k=0}^{K-1} \sum_{i=1}^p  \tau_{i}^{K-1}\left(  e_W^{\{k+i/p\}}\right).
\end{equation*}
Applying Corollary~\ref{corr:1} yields
\begin{equation*}
\left(e^{\{K\}},\psi\right) =  \sum_{i=1}^p \bilinhatenopar{i}{K-1}{K-1} -   \sum_{k=0}^{K-2} \sum_{i=1}^{p} \tau_{i}^{K-2}\left(e_W^{\{k + i/p\}}\right).
\end{equation*}
Repeated application of Corollary~\ref{corr:1} yields
\begin{equation}
\label{eq:main_thm_2nd_last}
\begin{aligned}
\left(e^{\{K\}},\psi\right) =  \sum_{k=0}^{K-1} \sum_{i=1}^{p} \bilinhatenopar{i}{k}{k}.
\end{aligned}
\end{equation}
Now,
\begin{equation}
\label{eq:main_thm_last}
\begin{aligned}
	\bilinhatenopar{i}{k}{k} &= a_i\left(u^{\{k+i/p\}} - U^{\{k+i/p\}}, \phi^{[{k+i/p}]}\right) =  a_i\left(\widetilde{u}^{\{k+i/p\}}, \phi^{[{k+i/p}]}\right) - a_i\left(\widetilde{U}^{\{k+i/p\}}, \phi^{[{k+i/p}]}\right) \\
	&= l_i \left(\phi^{[{k+i/p}]}\right)  - a_i\left(\widetilde{U}^{\{k+i/p\}}, \phi^{[{k+i/p}]}\right) = R_i\left(\widetilde{U}^{\{k+i/p\}},\phi^{[{k+i/p}]}\right).
	\end{aligned}
\end{equation}
Combining \eqref{eq:main_thm_2nd_last} and \eqref{eq:main_thm_last} leads to
\begin{equation}
\label{eq:err_rep_no_gal_ortho}
\left(\psi, u^{\{K\}} - U^{\{K\}}\right) = \sum_{k=0}^{K-1} \sum_{i=1}^p  R_i\left(\widetilde{U}^{\{k+i/p\}}, \phi^{[k+i/p]}\right).
\end{equation}
The discrete equivalent of \eqref{eq:multiplicative_basic_local} is
\begin{equation}
      \label{eq:multiplicative_basic_local_discrete}
      R_i\left(\widetilde{U}^{\{k+i/p\}}, v\right) =  l_i(v) - a_i \left( \widetilde{U}^{\{k+i/p\}},v \right),
            \quad \forall v \in V_{i,h,0}.
    \end{equation}
Substituting $v = \pi_i \phi^{[k+i/p]} \in  V_{i,h,0}$ in \eqref{eq:multiplicative_basic_local_discrete} and subtracting the result from \eqref{eq:err_rep_no_gal_ortho} completes the proof.

\end{proof}

\subsection{Details of analysis of additive Schwarz algorithm}
\label{sec:additive_Schwarz_error_analysis}

\subsubsection{Reformulation of the algorithm}
Similar to the multiplicative case in \S \ref{sec:multiplicative_Schwarz_error_analysis}, the basic additive algorithm \ref{alg:additive_basic} is not amenable to adjoint based analysis since the affine solution space $H^{1}_{D_{k}}(\Omega_i)$ changes at every iteration. We reformulate the algorithm  by again using a standard lifting technique to account for this. We set
\begin{equation}
\label{eq:def_w_add_sch}
\widetilde{u}_i^{\{k+1\}} = w_i^{\{k+1\}}  + u^{\{k\}} \quad \text{ on } \Omega_i,
\end{equation}
where now $w_i^{\{k+1\}} \in H_0^{1}(\Omega_i)$. This results in Algorithm~\ref{alg:additive_reformulated}.

 %TODO: add in Galerkin orthogonality

\begin{algorithm}[H]
\caption{Reformulated overlapping additive Schwarz}
\label{alg:additive_reformulated}

\begin{algorithmic}
\State Given $u^{\{0\}}$ defined on $\Omega$
\For{$k=0, 1, 2, \dots, K-1$ }
  \For{$i= 1, 2, \dots, p$ }
    % \State \\
    \State Find $w_i^{\{k+1\}} \in H_0^{1}(\Omega_1)$ such that \\
    \begin{equation}
   \label{eq:additive_def_uh}
     a_i \left( w_i^{\{k+1\}},v \right)= l_i(v)- a_i \left(u^{\{k \}},v \right),
         \quad \forall v \in H_0^{1}(\Omega_i).
   \end{equation}
   % \State \\
   \State Let
   \begin{equation}
   \label{eq:additive_ukph}
     u^{\{k+1\}} = u^{\{k\}} + \tau \left ( \sum_{i=1}^p  \widetilde{\Pi }_i w_i^{\{k+1\}} \right )
     \;\hbox{where}\quad
     \widetilde{\Pi }_i w_i^{\{k+1\}} = \left \{
       \begin{gathered} \begin{aligned}
         &w_i^{\{k+1\}}, \; &&\hbox{on } \overline{\Omega}_i, \\
         &0,                &&\hbox{on } \Omega \backslash \overline{\Omega}_i.
       \end{aligned} \end{gathered}
       \right .
   \end{equation}
  \EndFor
\EndFor
\end{algorithmic}

\end{algorithm}
%

%The discretization of the domain and the definition of the finite element spaces is the same as in \S \ref{sec:multiplicative_Schwarz_finite_element}.
%We represent the discrete solutions as $U^{\{k\}}$.

There is an equivalent reformulation of the discrete Algorithm~\ref{alg:additive_reformulated} and we denote the unknown solutions  as $W^{\{k\}}$  belonging to the spaces $V_{i,h,0} \subset H_0^1(\Omega_i)$. These solutions are defined formally but are not computed in practice. Equation \eqref{eq:additive_ukph}, which shows that $u^{\{k+1\}}$ is a weighted sum of all previous solutions to \eqref{eq:additive_def_uh}, results in very different adjoint problems for additive Schwarz (equations \eqref{eq:additive_adjoints}) from those for multiplicative Schwarz (equations \eqref{eq:multiplicative_adj} and \eqref{eq:def_tau_i_Q}).

To distinguish between different solutions (true, analytical, discrete) we  use the notation in Table~\ref{tab:additive_Schwarz_spaces}.
\begin{table}[H]
\centering
\begin{tabular}{||c|c|c|l||}
\hline
Notation &  Formula       &Space & Meaning \\
\hline
$u$                        && $H^1_0(\Omega)$              & True solution \\
$u^{\{k\}}$                && $H^{1}_{0}(\Omega_i)$        & Global analytic solution at iteration $k$ \\
${U}^{\{k\}}$              && $V_{h}$                      & Global discrete solution at iteration $k$ \\
\hline
$\widetilde{u}_i^{\{k\}}$  && $H^{1}_{D_{k}}(\Omega_i)$    & Analytic solutions on $\Omega_i$ at iteration $k$ \\
$\widetilde{U}_i^{\{k\}}$  && $V_{i,h}^{k}$                & Discrete solutions on $\Omega_i$ at iteration $k$ \\
\hline
$w^{\{k\}}_i$              && $H^{1}_{0}(\Omega_i)$        & Analytic solns on $\Omega_i$ with homogen. bcs at iteration $k$ \\
$W^{\{k\}}$                && $V_{i,h,0}$                  & Discrete solns on $\Omega_i$ with homogen. bcs at iteration $k$ \\
\hline
$e^{\{k\}}$         &$u-U^{\{k\}}$                    & $H^{1}_{0}(\Omega)$      & Total error \\
%$\bar{e}^{\{k\}} = u - {u}^{\{k\}}$                        & $H^{1}_{0}(\Omega)$      & Iteration error\\
%$\hat{e}^{\{k\}} = {u}^{\{k\}} - {U}^{\{k\}}$              & $H^{1}_{0}(\Omega)$      & Discretization error\\
%$\widetilde{e}_i^{\{k\}} = {w}_i^{\{k\}} - {W}_i^{\{k\}}$  & $H^{1}_{0}(\Omega_i)$    & Local domain error\\
$e_I^{\{k\}}$       &$u-u^{\{k\}}$                    & $H^{1}_{0}(\Omega)$      & Global iteration error at iteration $k$  \\
$e_D^{\{k\}}$       &$u^{\{k\}}-U^{\{k\}}$            & $H^{1}_{0}(\Omega)$      & Global discretization error at iteration $k$ \\
\hline
$e_{W,i}^{\{k\} }$  &$w_i^{\{k\}}-W_i^{\{k\}}$        & $H^{1}_{0}(\Omega_i)$    & Discretization error on $\Omega_i$ with homogen. bcs at it. $k$ \\
\hline
\end{tabular}
\caption{Additive Schwarz: notation for different solutions and their spaces. }
\label{tab:additive_Schwarz_spaces}
\end{table}

\subsubsection{Technical lemmas}
\label{sec:additive_Schwarz_lemmas}

Let $e_W^{\{k\}} = w^{\{k\}} - W^{\{k\}}$. By \eqref{eq:additive_ukph} we have
% for $k \geq 1$,
\begin{equation}
\label{eq:additive_err_tilde_hat_rel}
\begin{aligned}
e_D^{\{k\}} &= e_D^{\{k-1\}} + \tau \sum_{i=1}^p  \widetilde{\Pi}_i e_{W,i}^{\{k\}}.
\end{aligned}
\end{equation}
%
%
%Let $\chi_i$ be a partition of unity such that
%\begin{equation}
%\psi_i=\chi_i \psi
%\end{equation}
%and $\psi_i=0$ on $\Omega \backslash \Omega_i$.
We apply lemma \ref{lem:partition of unity} to arrive at
\begin{equation}
\label{eq:additive_err_pou_decomp}
\left(e_D^{\{k\}},\psi\right) =  \left(e_D^{\{k\}},\sum_{i=1}^p \chi_i \psi \right)  =\sum_{i=1}^p \left(e_D^{\{k\}},\psi_i\right)_{ii}.
\end{equation}

%%%%%%%%%%%%%%%%%%%%%%%%%%%%%%%%%%%%%%%%%%%%%%%%%%%%%%%%%%%%%%%%%%%%%%%%%%%%%%%%%%%%%%%%%%%%%%%%%%%%
%                                                                                                  %
%  Lemma Six                                                                                       %
%                                                                                                  %
%%%%%%%%%%%%%%%%%%%%%%%%%%%%%%%%%%%%%%%%%%%%%%%%%%%%%%%%%%%%%%%%%%%%%%%%%%%%%%%%%%%%%%%%%%%%%%%%%%%%

%{\tt SJT: We had a problem here due to change from the original notation, but I think I have fixed it.}

\begin{lemma}[Error in QoI in terms of discretization errors with homogeneous bcs]
\label{lem:additive_disc_err_decomp}

The discretization error in the QoI is
\begin{equation*}
\left(e_D^K,\psi\right) = \tau \ \sum_{k=1}^{K} \sum_{i=1}^p \sum_{j=1}^p \ \left( e_{W,i}^{\{k\}}, \psi_j \right)_{ij}.
\end{equation*}

\end{lemma}

\begin{proof}
Using \eqref{eq:additive_err_tilde_hat_rel}, we have for a fixed $j$
\begin{equation*}
\left(e_D^{\{K\}}, \psi_j\right)_{jj} = \left(e_D^{\{K-1\}}, \psi_j\right)_{jj} + \tau \sum_{i=1}^p \left( e_{W,i}^{\{K\}}, \psi_j \right)_{ij},
\end{equation*}

This is a recursive relation involving $\hat e^{\{K\}}$. Unrolling the recursion leads to
\begin{equation*}
\begin{aligned}
\left(e_D^{\{K\}}, \psi_j\right)_{jj}&= \tau \ \sum_{k=1}^K \sum_{i=1}^p \left( e_{W,i}^{\{k\}}, \psi_j \right)_{ij}.
\end{aligned}
\end{equation*}
Summing over all $j = 1, \dots, p$ and using \eqref{eq:additive_err_pou_decomp},
\begin{equation*}
\begin{aligned}
\left(\hat e^{\{K\}}, \psi\right)
&= \tau \ \sum_{k=1}^K \sum_{i=1}^p \sum_{j=1}^p \ \left(  e_{W,i}^{\{k\}}, \psi_j \right)_{ij}.                         .
\end{aligned}
\end{equation*}
%since all inner products are over subsets of subdomain $i$.
\end{proof}

%%%%%%%%%%%%%%%%%%%%%%%%%%%%%%%%%%%%%%%%%%%%%%%%%%%%%%%%%%%%%%%%%%%%%%%%%%%%%%%%%%%%%%%%%%%%%%%%%%%%
%                                                                                                  %
%  Lemma Seven                                                                                     %
%                                                                                                  %
%%%%%%%%%%%%%%%%%%%%%%%%%%%%%%%%%%%%%%%%%%%%%%%%%%%%%%%%%%%%%%%%%%%%%%%%%%%%%%%%%%%%%%%%%%%%%%%%%%%%

\begin{lemma} [Bilinear form with global discretization errors]
\label{lem:additive_disc_err_unrolling}
For any $v \in V_i$ we have
%\tred{\tt TODO: Check the eq below carefully}
\begin{equation*}
  a_i\left(e_D^{\{k\}},v\right) = \tau \sum_{m = 1}^{k}\sum_{j=1}^p a_{i j}\left(e_{W,j}^{\{m\}},v\right).
\end{equation*}
\end{lemma}

\begin{proof}
By \eqref{eq:additive_err_tilde_hat_rel}, we have
%\tred{\tt TODO: Check the eq below carefully}
\begin{equation*}
\begin{aligned}
a_i\left( e_D^{\{k\}},v\right)
&= a_i\left( e_D^{\{k-1\}},v\right) + \tau \sum_{j=1}^p a_{i j}\left(e_{W,j}^{\{k\}},v\right) \\
&= a_i\left( e_D^{\{k-1\}},v\right) + \tau \sum_{j=1}^p a_{i j}\left(e_{W,j}^{\{k\}},v\right) ,
\end{aligned}
\end{equation*}
since $e_W$ is the identity on subdomain $j$. This is a recursive relation involving $a_i\left( e_D^{\{k\}},v\right)$. Unrolling this recursion and using the fact that $e_D^{\{0\}} = 0$ proves the result.

\end{proof}

%%%%%%%%%%%%%%%%%%%%%%%%%%%%%%%%%%%%%%%%%%%%%%%%%%%%%%%%%%%%%%%%%%%%%%%%%%%%%%%%%%%%%%%%%%%%%%%%%%%%
%                                                                                                  %
%   Adjoint problems                                                                               %
%                                                                                                  %
%%%%%%%%%%%%%%%%%%%%%%%%%%%%%%%%%%%%%%%%%%%%%%%%%%%%%%%%%%%%%%%%%%%%%%%%%%%%%%%%%%%%%%%%%%%%%%%%%%%%

% \subsection{Adjoint problems for discretization error}
% \label{sec:additive_Schwarz_adjoint_problems}

% We define the adjoint problems,

% \begin{equation}\label{eq:additive_adjoints}
% a_{i}(v, \phi_i^{[k]})
% = \tau \sum_{j=1}^p \left[ (\psi_j,v)_{ij} - a_{ij}\left(v, \sum_{l=k+1}^K \phi_j^{[l]} \right) \right].
% \end{equation}

%%%%%%%%%%%%%%%%%%%%%%%%%%%%%%%%%%%%%%%%%%%%%%%%%%%%%%%%%%%%%%%%%%%%%%%%%%%%%%%%%%%%%%%%%%%%%%%%%%%%
%                                                                                                  %
%  Lemma Eight                                                                                     %
%                                                                                                  %
%%%%%%%%%%%%%%%%%%%%%%%%%%%%%%%%%%%%%%%%%%%%%%%%%%%%%%%%%%%%%%%%%%%%%%%%%%%%%%%%%%%%%%%%%%%%%%%%%%%%

\begin{lemma}[Bilinear form with local discretization errors with homogeneous bcs]
\label{lem:add_sch_err_ana_1}
\begin{equation*}
  a_i\left(e_{W,i}^{\{k\}}, \phi_i^{[k]}\right)
  = R_i\left(\widetilde{U}_i^{\{k\}},  \phi_i^{[k]}\right) - \tau \sum_{m = 1}^{k-1}\sum_{j=1}^p a_{i j}\left(e_{W,j}^{\{m\}},\phi_i^{[k]}\right).
\end{equation*}
\end{lemma}

\begin{proof}

By definition of $e_{W,i}^{\{k\}}$,
\begin{equation*}
\begin{aligned}
a_i\left(e_{W,i}^{\{k\}}, \phi_i^{[k]}\right)  &=
a_i\left(w_i^{\{k\}}, \phi_i^{[k]}\right) - a_i\left(W_i^{\{k\}}, \phi_i^{[k]}\right)\\
&= a_i\left(w_i^{\{k\}} + u_i^{\{k-1\}}, \phi_i^{[k]}\right) - a_i\left(u_i^{\{k-1\}}, \phi_i^{[k]}\right) -
a_i\left(W_i^{\{k\}} + U_i^{\{k-1\}}, \phi_i^{[k]}\right) + a_i\left(U_i^{\{k-1\}}, \phi_i^{[k]}\right).
\end{aligned}
\end{equation*}
Using \eqref{eq:def_w_add_sch} followed by \eqref{eq:additive_def_uh} and definition of $e_D^{\{k\}}$,
\begin{equation*}
\begin{aligned}
a_i\left(e_{W,i}^{\{k\}}, \phi_i^{[k]}\right)  &=
a_i\left( \widetilde{u}_i^{\{k\}}, \phi_i^{[k]}\right) - a_i\left(u_i^{\{k-1\}}, \phi_i^{[k]}\right) -
a_i\left(\widetilde{U}_i^{\{k\}}, \phi_i^{[k]}\right) + a_i\left(U_i^{\{k-1\}}, \phi_i^{[k]}\right)\\
&= R_i\left(\widetilde{U}_i^{\{k\}},\phi_i^{[k]}\right)
 - a_i\left(e_D^{\{k-1\}}, \phi_i^{[k]}\right).
\end{aligned}
\end{equation*}
By Lemma~\ref{lem:additive_disc_err_unrolling},
\begin{equation*}
\begin{aligned}
a_i\left(e_{W,i}^{\{k\}}, \phi_i^{[k]}\right)  &=
R_i\left(\widetilde{U}_i^{\{k\}},\phi_i^{[k]}\right)
 - \tau \sum_{m = 1}^{k-1}\sum_{j=1}^p a_{i j}\left(e_{W,j}^{\{m\}},\phi_i^{[k]}\right).
\end{aligned}
\end{equation*}

\end{proof}

%%%%%%%%%%%%%%%%%%%%%%%%%%%%%%%%%%%%%%%%%%%%%%%%%%%%%%%%%%%%%%%%%%%%%%%%%%%%%%%%%%%%%%%%%%%%%%%%%%%%
%                                                                                                  %
%   Theorem: Discretization error representation for additive Schwarz                              %
%                                                                                                  %
%%%%%%%%%%%%%%%%%%%%%%%%%%%%%%%%%%%%%%%%%%%%%%%%%%%%%%%%%%%%%%%%%%%%%%%%%%%%%%%%%%%%%%%%%%%%%%%%%%%%

% \begin{theorem}[Discretization error for additive Schwarz]
% \label{thm:additive_discretization_error}

% \begin{equation}
% \label{eq:additive_disc_err_rep}
% (\psi, e_D^{\{K\}}) = (\psi, u^{\{K\}} - U^{\{K\}})  = \sum_{k=1}^{K} \sum_{i=1}^p  R_i(\widetilde{U}^{\{k\}}, \phi^{[k]}).
% \end{equation}

% \end{theorem}

\subsubsection{Proof of Theorem \ref{thm:additive_discretization_error}}

\begin{proof}
By \eqref{eq:additive_adjoints},
\begin{equation*}
\left(\psi, e_D^{\{K\}}\right)
=  \tau \sum_{k=1}^{K} \sum_{i = i}^p \sum_{j = 1}^p  \left(\psi_j,e_{W,i}^{\{k\}}\right)_{ij}
= \sum_{k=1}^{K} \sum_{i=1}^p
\left\{
a_{i}\left(e_{W,i}^{\{k\}}, \phi_i^{[k]}\right)
    + \tau \sum_{j=1}^p \sum_{l=k+1}^K a_{ij}\left(e_{W,i}^{\{k\}},  \phi_j^{[l]} \right\}
\right].
\end{equation*}
By Lemma~\ref{lem:add_sch_err_ana_1},
\begin{equation*}
\left(\psi, e_D^{\{K\}}\right)
= \sum_{k=1}^{K}  \sum_{i=1}^p
\left\{
R_i\left(\widetilde{U}_i^{\{k\}}, \phi_i^{[k]}\right)
- \tau \sum_{m = 1}^{k-1}\sum_{j=1}^p
    a_{i j}\left(e_{W,j}^{\{m\}},\phi_i^{[k]}\right)
      + \tau \sum_{j=1}^p \sum_{l=k+1}^K a_{ij} \left(e_{W,i}^{\{k\}},  \phi_j^{[l]} \right)
\right\}.
\end{equation*}
Application of Galerkin orthogonality, similar to its use in the proof in \S \ref{sec:proof_mult_schwarz_disc_err}, leads to
\begin{equation*}
\left(\psi, e_D^{\{K\}}\right)
= \sum_{k=1}^{K}  \sum_{i=1}^p
\left\{
R_i\left(\widetilde{U}_i^{\{k\}}, \phi_i^{[k]} - \pi_i \phi_i^{[k]}\right)
- \tau \sum_{m = 1}^{k-1}\sum_{j=1}^p
    a_{i j}\left(e_{W,j}^{\{m\}},\phi_i^{[k]}\right)
      + \tau \sum_{j=1}^p \sum_{l=k+1}^K a_{ij} \left(e_{W,i}^{\{k\}},  \phi_j^{[l]} \right)
\right\}.
\end{equation*}
The result follows if
\begin{equation*}
\sum_{k=1}^{K}  \sum_{j = 1}^p \sum_{m = 1}^{k-1}\sum_{i=1}^p a_{i j}\left(e_{W,i}^{\{m\}},\phi_j^{[k]}\right)
= \sum_{k=1}^{K}  \sum_{i = i}^p \sum_{j = 1}^p \sum_{l=k+1}^K a_{ij}\left(e_{W,i}^{\{k\}},  \phi_j^{[l]} \right),
\end{equation*}
where we interchanged the $i$ and $j$ loop indices on the left hand side. This follows if
\begin{equation}
\label{eq:two_sums}
  \sum_{k=1}^{K} \sum_{m=1}^{k-1}  a_{i j}\left(e_{W,i}^{\{m\}},\phi_j^{[k]}\right)
= \sum_{k=1}^{K} \sum_{l=k+1}^K a_{ij} \left(e_{W,i}^{\{k\}},  \phi_j^{[l]} \right).
\end{equation}
To see why this is true, let $A$ be a $K \times K$ strictly lower triangular matrix where the non-zero entries are given by $A_{k,m} = a_{i j}\left(e_{W,i}^{\{m\}},\phi_j^{[k]}\right) $ for $m < k$. Then the left hand side of \eqref{eq:two_sums} is the sum of the entries of $A$ by first summing each row while the right hand side of \eqref{eq:two_sums} is the sum of the entries of $A$ by first summing each column.

%TODO: add in Galerkin orthogonality

\end{proof}

\section{Conclusions and future directions}
\label{sec:discussion}

We develop an adjoint based \emph{a posteriori} error analysis to evaluate the  discretization and iteration errors for a given quantity of interest when solving boundary value problems using overlapping domain decomposition employing either multiplicative or additive Schwarz iteration. The additional expense of formulating and solving the necessary sequence of adjoint problems both recommends and enables a two stage approach to constructing efficient solution strategies. In this approach, a ``stage 1'' solution is computed on a relatively coarse discretization employing a small number of iterations or small overlap between subdomains. The error in the quantity of interest is determined for the stage 1 solution and the balance of discretization and iteration errors, and the distribution of discretization error between subdomains, is determined. These guide the solution strategy for a more accurate ``stage 2'' solution in terms of the localized refinement of the finite element mesh and the choices of overlap and number of iterations.

The adjoint based analysis in this article has focused exclusively on linear problems. Adjoint based analysis can be extended to nonlinear problems, see \cite{marchuk1, estep2009error}.
%
% There is no unique definition of an adjoint operator to a nonlinear differential equation, but a common choice useful for various kinds of analysis is based on linearization \cite{marchuk1}. Considering equation \eqref{eq:diff_eq_generic} again, where $L$ is now a nonlinear operator and setting $z = su + (1-s)U$, we define the linearized adjoint operator to be
% \begin{equation}\label{linearization}
%   (DL)^\ast = \left [ \int_0^1 \frac{\partial L}{\partial u}\bigg|_z ds \right]^\ast,
% \end{equation}
% and the corresponding linearized adjoint problem to be
% \begin{equation}\label{formallinearizedadjoint}
% (DL)^\ast \phi = \psi.
% \end{equation}
% %where $DL$ refers to the linearization of $L(u)$ evaluated at $z$.
% Formally, we use the solution of \eqref{formallinearizedadjoint} in the {\em a posteriori} error analysis to obtain \eqref{eq:generic_err_eq_simple}, but in practice, we linearize around the numerical solution. The resulting estimate can be shown to converge to the true estimate in the limit of refined discretization \cite{estlarwil00} and yields robustly accurate error estimates.
A consideration of nonlinear problems is therefore an obvious and relatively immediate extension of this work.

A more serious extension is to address initial boundary value problems. In combination with earlier work on parallel methods for initial value problems \cite{chaudhry2016posteriori}, the current analysis should enable the development of an \emph{a posteriori} analysis for a numerical method that is parallel in both space and time. Such an analysis would again enable an efficient two stage solution approach, using the distribution of various sources of error estimated from an initial coarse solution to inform the discretization choices for a second ``production'' computation.

%We will also analyse non-overlapping domain decomposition such as FETI~\cite{FaRo1991,FMR1994} and BDDC~\cite{FLLPR2001,MaDo2003} methods by extending work on Lions domain decomposition~\cite{JBE18} and  mortar elements \cite{arbogast2014posteriori, arbogast2015posteriori} to develop distinct estimates of the errors arising on domain interiors and interfaces. This is particularly appropriate when the quantity of interest, for example a flux, is defined on the boundary between two or more domains. %Finally, domain decomposition methods are commonly employed as preconditioners for multigrid methods. The question arises as to whether a two-stage approach can be developed where not only is the local refinement of the finite element mesh and also the disposition of the subdomains is chosen not only to improve the accuracy of the solution but also to enhance the properties of the preconditioner.

\section*{Acknowledgments}
J. Chaudhry’s work is supported by the NSF-DMS 1720402.
S. Tavener's work is supported by NSF-DMS 1720473.
D. Estep's work is supported by NSF-DMS 1720473.

\bibliographystyle{plain}
\bibliography{refs_schwarz_analysis}

%\appendix
%\input{Appendix_GaussSeidel}

\end{document}